\documentclass[10pt]{article}
   \usepackage{graphicx}
   \usepackage{epsfig}
   \usepackage{amssymb}
   \usepackage{amsmath}
   \usepackage{amsthm}
   \newtheorem{lemma}{Lemma}[section]
   \newtheorem{theorem}{Theorem}[section]
   
   \newtheorem{corollary}{Corollary}[section]
   \newtheorem{remark}{Remark}[section]

   {\end{description}}

   {\hfill $\bullet$ \\}

   \newcommand{\be}{\begin{equation}}
   \newcommand{\ee}{\end{equation}}

\begin{document}
    \title{An efficient predictor-corrector approach with orthogonal spline collocation finite element technique for FitzHugh-Nagumo problem}
  \author{Eric Ngondiep\thanks{\textbf{Correspondence to:} Eric Ngondiep, ericngondiep@gmail.com/engondiep@imamu.edu.sa}}
   \date{\small{Department of Mathematics and Statistics, College of Science, Imam Mohammad Ibn Saud\\ Islamic University
        (IMSIU), $90950$ Riyadh $11632,$ Saudi Arabia.}}
   \maketitle

   \textbf{Abstract.}
   This paper constructs a predictor-corrector technique with orthogonal spline collocation finite element method for simulating a FitzHugh-Nagumo system subject to suitable initial and boundary conditions. The developed computational technique approximates the exact solution in time using variable time steps at the predictor phase and a constant time step at the corrector stage while the orthogonal spline collocation finite element method is employed in the space discretization. The new algorithm presents several advantages: (i) the errors increased at the predictor phase are balanced by the ones decreased at the corrector phase so that the stability is preserved, (ii) the variable time steps at the predictor stage overcome the numerical oscillations, (iii) the spatial errors are minimized due to the use of collocation nodes, and (iv) the linearization of the nonlinear term reduces the required operations at the corrector stage. As a result, the new computational technique computes efficiently both predicted and corrected solutions and preserves a strong stability and high-order accuracy even in the presence of singularities. The theoretical results suggest that the constructed approach is unconditionally stable, spatial $mth$-order accurate and temporal second-order convergent in the $L^{\infty}(0,T;[H^{m}(\Omega)]^{2})$-norm. Some numerical experiments are carried out to confirm the theoretical analysis and to demonstrate the applicability and performance of the proposed strategy.\\
    \text{\,}\\

   \\
   {\bf AMS Subject Classification (MSC). 65M12, 65M15, 65M60, 65M70}.

  \section{Introduction}\label{sec1}
    The FitzHugh-Nagumo (FHN) model is a system of reaction-diffusion equations describing self-excitation via nonlinear positive feedback and recovery through negative feedback \cite{6wcqh,4wcqh,2wcqh}. This model is a simplified representation of the Hodgkin-Huxley problem \cite{8wcqh,7wcqh} and it is considered as a numerically tractable excitation model to study the propagation of phenomena where the behavior of the excitation wave is more suitable than the shape of the action potential \cite{10wcqh,15wcqh}. Both FitzHugh-Nagumo and Hodgkin-Huxley models have certain similar qualitative behaviors. For some parameter regimes, the model presents a globally stable resting state although a sufficiently strong perturbation will cause a spike. Additionally, the system lacks a saddle node and so exhibits a pseudo-threshold that is compatible to the Hodgkin-Huxley system \cite{5wcqh}. The FHN problem can be enhanced to generate longer action potentials and eliminate hyperpolarisation during the depolarisation stage \cite{5wcqh}. However, this model lies in the class of nonlinear time-dependent partial differential equations (PDEs) whose the computation of an exact solution is very difficult and sometimes impossible. For more details about such a set of equations, the interested readers can consult \cite{2en,4dm,3en,18wcqh,4en} and references therein. The first challenge when computing an analytical solution of a FHN problem follows from the coupling unknown functions and the nonlinear reaction term. In this paper, we are interested in an approximate solution of the FitzHugh-Nagumo system defined in \cite{8dm} as\\

     \begin{equation}\label{1}
     \left\{
       \begin{array}{ll}
         \frac{\partial u}{\partial t}-\gamma_{1}\Delta u=f_{1}(u)-g(u,v), & \hbox{on $\Omega\times[0,\text{\,}T]$} \\
     \text{\,}\\
         \frac{\partial v}{\partial t}-\gamma_{2}\Delta v=f_{2}(u,v),  & \hbox{on $\Omega\times[0,\text{\,}T]$} \\
       \end{array}
     \right.
     \end{equation}
     subject to initial conditions
      \begin{equation}\label{2}
      u(x,0)=u_{0}(x),\text{\,\,\,}v(x,0)=v_{0}(x),\text{\,\,\,\,on\,\,\,\,}\overline{\Omega}=\Omega\cup\Gamma,
     \end{equation}
     and boundary conditions
      \begin{equation}\label{3}
      \frac{\partial u}{\partial \eta}=-\beta_{1}u,\text{\,\,\,}\frac{\partial v}{\partial \eta}=-\beta_{2}v,\text{\,\,\,\,on\,\,\,\,}\Gamma\times[0,\text{\,}T],
     \end{equation}
     where: $"\Delta"$ denotes the Laplacian operator, $T>0$ is the final time, $\Omega\subset\mathbb{R}^{d}$ ($d=1,2,3$) is an open bounded convex polygonal domain while $\Gamma$ designates its boundary. $\beta_{k}$ and $\gamma_{k}$, for $k=1,2,$ are nonnegative physical parameters, $u$ is the fast membrane potential of a neuron whereas $v$ designates the slow recovery unknown function that reflects a combination of potassium channel activation and sodium channel inactivation effects, $u_{0}$ and $v_{0}$ represent the initial conditions. $f_{2}$ and $g$ are the linear reaction terms defined as $f_{2}(u,v)=\theta_{1}u-\theta_{2}v$ and $g(u,v)=v$, whereas $f_{1}$ is the nonlinear reaction term given by $f_{1}(u)=u(1-u)(u-\theta_{3})$, where $\theta_{1}>0$, $\theta_{2}\geq0$, and $0<\theta_{3}<1$.\\

Several authors have developed a broad range of numerical methods for solving the FHN system $(\ref{1})$-$(\ref{3})$ and its simplified form \cite{16wcqh}, such as: finite
difference schemes and implicit-explicit with finite element methods \cite{8dm,31wcqh,23wcqh}, pseudspectral methods \cite{21wcqh,16wcqh}, variable time-step with
two-grid finite element formulation \cite{24wcqh} and mixed two-grid difference approach \cite{20wcqh}. In this work, we construct an efficient predictor-corrector technique with orthogonal spline collocation finite element method in a computed solution of the initial-boundary value problem $(\ref{1})$-$(\ref{3})$. The proposed strategy discretizes the time derivative using interpolation approach with variable time steps at the predictor stage and a constant time step at the corrector stage while the space derivatives are approximated with the use of the orthogonal spline collocation finite element method. The new algorithm is unconditionally stable, temporal second-order accurate, space mth-order convergent (where $m$ denotes an integer greater than or equal $3$), efficient and takes many advantages compared to a wide set of numerical schemes discussed in the literature for solving a general class of FitzHugh-Nagumo equations and other sets of time-dependent PDEs \cite{8dm,24wcqh,23wcqh,34wcqh,25wcqh}. Indeed, the errors increased at the predictor phase are balanced by the ones decreased at the corrector phase so that the stability of the developed computational technique is preserved. In addition, the use of nonuniform time steps at the predictor stage greatly reduces the numerical oscillations (due to the convection terms in the weak formulation and nonlinear reaction term) as has been observed when applying the Gaussian quadrature formula for numerical integration problems, whereas the orthogonal spline collocation FEM extracts and retains the essential characteristics of the computed solution within a finite-dimensional space, thus significantly reduces the computational costs. Moreover, the approximations given by the orthogonal spline collocation FEM deal with both the solution and its space derivatives throughout a computational domain, and thus achieving spatial high-accuracy \cite{28wcqh,25wcqh,27wcqh,34wcqh}. Finally, the linearization of the nonlinear reaction term overcomes the difficulties to calculate the solution at the corrector step. With this linearization, this stage yields a block system of linear equations which are easily solved by computing the inverse of the coefficients matrix. This greatly reduces the time consuming of the algorithm. We remind that the highlights of this paper are the following items:
     \begin{description}
      \item[(i)] development of the predictor-corrector scheme with orthogonal spline collocation approach in a numerical solution of the FitzHugh-Nagumo system $(\ref{1})$ with initial-boundary conditions $(\ref{2})$-$(\ref{3})$,
      \item[(ii)] a detailed analysis of the stability and error estimates of the constructed numerical method,
      \item[(iii)] numerical examples to confirm the theoretical studies and to demonstrate the applicability and efficiency of the new computational technique.
     \end{description}

      The remainder of the paper is organized as follows. Section $\ref{sec2}$ presents a detailed description of the predictor-corrector with orthogonal spline collocation
      technique for solving the FitzHugh-Nagumo system $(\ref{1})$ with initial conditions $(\ref{2})$ and boundary conditions $(\ref{3})$. In Section $\ref{sec3}$,
     we provide the stability analysis and
     error estimates of the constructed numerical approach, while some computational examples are performed and discussed in Section $\ref{sec4}$ to confirm the theory and to
     establish the validity and efficiency of the new algorithm. Finally, the general conclusions and our future works are drawn in Section $\ref{sec5}$.

    \section{Construction of the new computational technique}\label{sec2}
    In this section, we develop a predictor-corrector scheme combined with an orthogonal spline collocation FEM for solving the FitzHugh-Nagumo model $(\ref{1})$ with initial and boundary conditions $(\ref{2})$ and $(\ref{3})$, respectively.\\

    Let $M$ and $N$ be two positive integers, and $\mathcal{F}_{h}=\{Q_{i},\text{\,}i=1,2,...,M\}$, be a conforming partition of the domain $\overline{\Omega}$ consisting of triangles/tetrahedra, where $h=\max\{|Q_{i}|,\text{\,}i=1,...,M\}$ and $|Q_{i}|$ denotes the diameter of $Q_{i}$. We assume that $\mathcal{F}_{h}$ satisfies the following properties: (a) $int(Q_{i})\neq\emptyset$, for $i=1,...,M$; (b) $int(Q_{i})\cap int(Q_{j})=\emptyset$, while $Q_{i}\cap Q_{j}=\emptyset$, if $|i-j|\geq p$ and $Q_{i}\cap Q_{j}=e$, whenever $0<|i-j|<p$ ($p$ is a nonnegative fixed integer), for $i,j=1,...,M$, where $e$ is the common face/edge. Additionally, there are elements $\bar{Q}_{i_{k}}\in\mathcal{F}_{h}$, for $k=1,...,p$, so that $Q_{i}\cap Q_{i+1}\cap(\underset{k=1}{\overset{p}\bigcap}\bar{Q}_{i_{k}})=e$ and $Q_{i}\cup Q_{i+1}\cup(\underset{k=1}{\overset{p}\bigcup}\bar{Q}_{i_{k}})$ is a convex set, for $i=1,...,M-(p+1)$; (c) the triangulation $\mathcal{F}_{\Gamma h}$ induced on $\Gamma$ is quasi-uniform.\\

For any edge/face $e=Q_{i}\cap Q_{j}$, $n_{\Gamma_{k}}$, for $k\in\{i,j\}$, represent the unit outward normal vectors to $e$ associated with $\Gamma_{k}=\partial Q_{k}$. The average and jump of an element $u$ are denoted by $\{\{u\}\}$ and $[[u]]$, respectively. They are defined as
\begin{equation}\label{e1}
 \{\{u\}\}=\frac{1}{2}(u|_{\Gamma_{i}}+u|_{\Gamma_{j}})\text{\,\,\,\,and\,\,\,\,}[[u]]=(u|_{\Gamma_{i}})n_{\Gamma_{i}}+(u|_{\Gamma_{j}})n_{\Gamma_{j}}.
\end{equation}

It's worth recalling that if $u$ is continuous on $\Omega$, then $\{\{u\}\}=u$ and $[[u]]=0$. In addition, if $e=Q_{i}\cap\Gamma$, then
\begin{equation}\label{e2}
 \{\{u\}\}=u|_{\Gamma_{i}}\text{\,\,\,\,and\,\,\,\,}[[u]]=(u|_{\Gamma_{i}})n_{\Gamma_{i}}.
\end{equation}

We introduce the sets $\mathcal{U}_{h}^{(k)}$ and $\mathcal{W}_{h}$ of all piecewise polynomials defined on the domain $\Omega$.
\begin{equation*}
\mathcal{U}_{h}^{(k)}=\{u_{h}\in H^{1}(\Omega)\cap L_{\sup}^{2}(\Omega):\text{\,}u_{h}|_{Q_{i}}\in\mathcal{P}_{m}(Q_{i}),\text{\,}[[\nabla u_{h}]]=0,\text{\,if\,}\Gamma_{i}\cap\Gamma=\emptyset,\text{\,and\,}[[\nabla u_{h}]] =-\beta_{k}u_{h},
\end{equation*}
\begin{equation}\label{4}
\text{\,on\,}\Gamma_{i}^{b}=\Gamma_{i}\cap\Gamma\neq\emptyset,\text{\,for\,}i=1,...,M\},
\end{equation}
for $k=1,2$, where $\mathcal{P}_{m}(Q_{i})$ is the set of all polynomials defined on $Q_{i}\subset\overline{\Omega}$, with degree at most $m$, and $L_{\sup}^{2}(\Omega)$ is the bounded vector space of square integrable functions defined over the domain $\Omega$. That is, there exists a positive constant $C_{\sup}$ so that: $\sup\{\|v\|_{L^{2}}:\text{\,}v\in L_{\sup}^{2}(\Omega)\}\leq C_{\sup}$. Set
\begin{equation}\label{4a}
\mathcal{W}_{h}=\mathcal{U}_{h}^{(1)}\times\mathcal{U}_{h}^{(2)}.
\end{equation}

It is not difficult to show that $\mathcal{W}_{h}$ is a $M_{m}$-dimensional vector space with $M_{m}\leq\frac{2M(m+d)!}{m!d!}$. Indeed $\mathcal{U}_{h}^{(k)}$ is a vector space of dimension less than $\frac{M(m+d)!}{m!d!}$. Consider $S=\{\alpha_{l}:\text{\,}l=1,2,...,L\}$, where $L\geq2$, $0=\alpha_{1}<...<\alpha_{L}<1$ or (or $0<\alpha_{1}<...<\alpha_{L}=1$), be the Gaussian collocation nodes on $[0,\text{\,}1]$, with associated weights: $c_{l}$, $l=1,...,L$. Since $\Omega$ is a convex set in $\mathbb{R}^{d}$, the sets $G_{i,l}=(1-\alpha_{l})Q_{i}+\alpha_{l}Q_{i+1}$, for $i=1,...,M-1$, and $l=1,...,L$, are contained in $\overline{\Omega}$. As a result, the set
\begin{equation}\label{5}
    \mathcal{C}_{h}=\{G_{i,l}:\text{\,}i=1,...,M-1, \text{\,}l=1,...,L\}
\end{equation}
is a collection of "Gaussian collocation sets" in $\Omega$. It is easy to see that $\mathcal{F}_{h}\subset\mathcal{C}_{h}$.\\

Let $\Pi_{N}=\{t_{n},\text{\,}n=0,1,...,N\}$, be a partition of the time interval $[0,\text{\,}T]$ (where $0=t_{0}<t_{1}<...<t_{N}=T$) such that: $t_{n+1}<\frac{1}{2}(t_{n}+t_{n+2})$, the sequence of local time steps $\tau_{s}=t_{s+\frac{1}{2}}-t_{s}$, for $s=n,n+\frac{1}{2}$, where $t_{n+\frac{1}{2}}=\frac{1}{2}(t_{n}+t_{n+1})$, is nondecreasing and satisfying $\tau_{s}\tau_{0}^{-1}\leq\hat{C}$, with $\hat{C}\geq1$ is a constant independent of the local time steps $\tau_{s}$ and space size $h$. It is not hard to see that $\tau_{n+\frac{1}{2}}-\tau_{n}=0$ and $\tau_{n}-\tau_{n-\frac{1}{2}}>0$, for $n=1,2,...,N$ (for example, the sequence $\{t_{n}=\frac{T}{\exp(1)-1}(\exp({\frac{n}{N}})-1),\text{\,}n=0,1,...,N\}$, satisfies such assumptions). For the convenience of writing, we set $w^{n}=w(x,t_{n})$ and $w_{h}^{n}=w_{h}(x,t_{n})$, be the value of the exact solution and approximate one, respectively, at point $(x,t_{n})$.\\

We define the set of all boundary edges contained in $\Gamma$.
 \begin{equation}\label{15}
    \mathcal{C}_{\Gamma h}=\{\Gamma^{b}_{i,l}=\Gamma_{i,l}\cap\Gamma\neq \emptyset:\text{\,}i=1,...,M-1,\text{\,}l=1,...,L\},
\end{equation}
where $\Gamma_{i,l}$ represents the boundary of the collocation set $G_{i,l}$. Further, we consider the sets of two-indices
 \begin{equation}\label{16}
    I=\{(i,l):\text{\,}G_{i,l}\in\mathcal{C}_{h}\},\text{\,\,\,\,and\,\,\,\,}I_{\Gamma}=\{(i,l):\text{\,}\Gamma_{i,l}^{b}\in\mathcal{C}_{\Gamma h}\}.
\end{equation}

 It is easy to see that $\underset{(i,l)\in I_{\Gamma}}{\bigcup}\Gamma^{b}_{i,l}=\Gamma$. Because $G_{i,l}\subset Q_{i}\cup Q_{i+1}\cup(\underset{k=1}{\overset{p}\bigcup}\bar{Q}_{i_{k}})$, without loss of generality, we assume that $\Gamma_{i,l}\subset \Gamma_{i}\cup \Gamma_{i+1}\cup(\underset{k=1}{\overset{p}\bigcup}\bar{\Gamma}_{i_{k}})$, for every $(i,l)\in I$, where $\bar{\Gamma}_{i_{k}}$ means the boundary of $\bar{Q}_{i_{k}}$.

     We introduce the linear operators $\bar{\Delta}$ and $\bar{\nabla}$, and bilinear operators $<\cdot,\cdot>$ and $<\cdot,\cdot>_{\cdot}$, together with discrete scalar products $\left(\cdot,\cdot\right)_{\cdot}$ and $\left(\cdot,\cdot\right)_{*,\cdot}$, defined as:
     \begin{equation*}
     \bar{\Delta}U=(\Delta u_{1},\Delta u_{2})^{t},\text{\,\,}\bar{\nabla}V=[\nabla v_{1},\nabla v_{2}],\text{\,\,}<U,V>=\underset{(i,l)\in I_{\Gamma}}{\sum}c_{l}
\int_{\Gamma_{i,l}^{b}}V^{t}(\bar{\nabla}^{t}U)n_{\Gamma_{i,l}^{b}}d\Gamma,
     \end{equation*}
     \begin{equation*}
     <U,V>_{\cdot}=\underset{(i,l)\in I\setminus I_{\Gamma}}{\sum}c_{l}\int_{\Gamma_{i,l}}V^{t}(\bar{\nabla}^{t}U)n_{\Gamma_{i,l}}d\Gamma_{i,l}=
     \underset{(i,l)\in I\setminus I_{\Gamma}}{\sum}c_{l}\underset{e\in\Gamma_{i,l}}{\sum}\int_{e}V^{t}(\bar{\nabla}^{t}U)n_{e}de=
     \end{equation*}
     \begin{equation*}
     \underset{(i,l)\in I\setminus I_{\Gamma}}{\sum}c_{l}\underset{e\in\Gamma_{i,l}}{\sum}\int_{e}V^{t}[[\bar{\nabla}^{t}U]]de=
     \underset{(i,l)\in I\setminus I_{\Gamma}}{\sum}c_{l}\int_{\Gamma_{i,l}}V^{t}[[\bar{\nabla}^{t}U]]d\Gamma_{i,l},
     \end{equation*}
     \begin{equation}\label{6}
     \left(U,V\right)_{\cdot}=\underset{(i,l)\in I}{\sum}c_{l}\int_{G_{i,l}}U^{t}Vdx,\text{\,\,}\left(\bar{\nabla}U,\bar{\nabla}V\right)_{*,\cdot}=
     \underset{(i,l)\in I}{\sum}\underset{j=1}{\overset{2}\sum}c_{l}\int_{G_{i,l}}(\nabla^{t}u_{j})(\nabla v_{j})dx,
     \end{equation}
for every vector-valued functions $U=(u_{1},u_{2})^{t}$ and $V=(v_{1},v_{2})^{t}$, defined on $\mathcal{C}_{h}$, where $n_{\Gamma_{i,l}^{b}}$, $n_{\Gamma_{i,l}}$ and $n_{e}$ are outward unit normal vectors to $\Gamma_{i,l}^{b}$, $\Gamma_{i,l}$ and $e$, respectively. Without loss of generality, the inner product $\left(\cdot,\cdot\right)_{\cdot}$, is also defined as in equation $(\ref{6})$, for every real-valued functions $u$ and $v$ defined on $\mathcal{C}_{h}$. The norms $\|\cdot\|_{\cdot}$ and $\|\cdot\|_{*,\cdot}$ associated with the scalar products $\left(\cdot,\cdot\right)_{\cdot}$ and $\left(\cdot,\cdot\right)_{*,\cdot}$ are given by
     \begin{equation}\label{7}
    \|U\|_{\cdot}=\sqrt{\left(U,U\right)_{\cdot}}\text{\,\,\,and\,\,\,}\|\bar{\nabla}U\|_{*,\cdot}=\sqrt{\left(\bar{\nabla}U,\bar{\nabla}U\right)_{*,\cdot}}.
     \end{equation}

     We will use the following integration by parts
     \begin{equation*}
    \left(\bar{\Delta}U,V\right)_{\cdot}=\underset{(i,l)\in I}{\sum}c_{l}\int_{G_{i,l}}\bar{\Delta}U^{t}Vdx=\underset{(i,l)\in I}{\sum}c_{l}\left[\int_{\Gamma_{i,l}}V^{t}(\bar{\nabla}^{t}U)n_{\Gamma_{i,l}}d\Gamma_{i,l}-\underset{j=1}{\overset{2}\sum}\int_{G_{i,l}}(\nabla^{t}u_{j})(\nabla v_{j})dx\right]=
     \end{equation*}
     \begin{equation}\label{8}
      <U,V>+<U,V>_{\cdot}-\left(\bar{\nabla}U,\bar{\nabla}V\right)_{*,\cdot}.
     \end{equation}

 Since $\mathcal{W}_{h}$ is a vector space of dimension $M_{m}$, let $\{\rho_{k},\text{\,}k=1,2,...,M_{m}\}$, be an orthogonal basis of $\mathcal{W}_{h}$ with respect to the scalar product $\left(\cdot,\cdot\right)_{\cdot}$. In fact, such a basis may be constructed by the use of the gram Schmidt orthogonalisation process that satisfies a three-term recurrence relation (similar to univariate Gram-Schmidt approach). Thus, for every $U_{h}\in\mathcal{W}_{h}$, there is a unique collection $(u_{h,1},u_{h,2},...,u_{h,M_{m}})$, so that
\begin{equation}\label{9}
      U_{h}|_{G_{i,l}}=\underset{k=1}{\overset{M_{m}}\sum}u_{h,k}.*\rho_{k}|_{G_{i,l}},
     \end{equation}
for any $i=1,2,...,M-1$, and $l=1,2,...,L$, where $".*"$ means the componentwise multiplication and $F|_{G_{i,l}}$ denotes the restriction of a vector-valued function on $G_{i,l}$.\\

Let $q_{h}$ be the $L^{2}$-projection from $L^{2}(\Omega)$ onto $\mathcal{U}_{h}^{(k)}$, for $k=1,2$. Set $P_{h}=(q_{h},q_{h})$ and $\mathbb{P}_{h}=(P_{h},P_{h})$, be the orthogonal projections from $[L^{2}(\Omega)]^{2}$ onto $\mathcal{W}_{h}$ and $[L^{2}(\Omega)]^{2}\times[L^{2}(\Omega)]^{2}$ onto $\mathcal{W}_{h}\times\mathcal{W}_{h}$, respectively. These projections are defined as: for every $Z=(U,V)^{t}=((u_{1},u_{2}),(v_{1},v_{2}))^{t}\in[L^{2}(\Omega)]^{4}$, $P_{U}=(q_{h}u_{1},q_{h}u_{2})^{t}$ and $\mathbb{P}_{h}Z=(P_{h}U,P_{h}V)^{t}$ and satisfy:

   \begin{equation}\label{10}
   \left(q_{h}u_{1},u_{1h}\right)_{\cdot}=\left(u_{1},u_{1h}\right)_{\cdot},\text{\,} \left(P_{h}U,U_{h}\right)_{\cdot}=\left(U,U_{h}\right)_{\cdot},\text{\,} \left(\mathbb{P}_{h}Z,Z_{h}\right)_{*,\cdot}=\left(\mathbb{P}_{h}Z,Z_{h}\right)_{*,\cdot}.
     \end{equation}
 for every $u_{1h}\in \mathcal{U}_{h}^{(k)}$, $U_{h}\in\mathcal{W}_{h}$, and $Z_{h}\in \mathcal{W}_{h}\times\mathcal{W}_{h}$. More specifically,
\begin{equation*}
   \int_{G_{i,l}}(P_{h}U)^{t}U_{h}dx=\int_{G_{i,l}}U^{t}U_{h}dx\text{\,\,\,and\,\,\,}\underset{j=1}{\overset{2}\sum}\int_{G_{i,l}}P_{h}(\nabla u_{j})^{t}\nabla v_{h,j}dx=
\underset{j=1}{\overset{2}\sum}\int_{G_{i,l}}(\nabla u_{j})^{t}\nabla v_{h,j}dx,\text{\,\,\,for\,\,\,}(i,l)\in I.
     \end{equation*}

   Now, set $w=(u,v)^{t}$, $F(w)=(f_{1}(u)-g(u,v),f_{2}(u,v))^{t}$, $\bar{\gamma}=(\gamma_{1},\gamma_{2})^{t}$, $\bar{\beta}=(\beta_{1},\beta_{2})^{t}$, and $w_{0}=(u_{0},v_{0})^{t}$. Utilizing this, the initial-boundary value problem $(\ref{1})$-$(\ref{3})$ can be written in vector form as
          \begin{equation}\label{11}
         \frac{\partial w}{\partial t}-\bar{\gamma}.*\bar{\Delta} w=F(w), \text{\,\,\,on\,\,\,}\Omega\times[0,\text{\,}T],
     \end{equation}
     subject to initial conditions
      \begin{equation}\label{12}
      w(0)=w_{0},\text{\,\,\,\,on\,\,\,\,}\overline{\Omega},
     \end{equation}
     and boundary conditions
      \begin{equation}\label{13}
      \frac{\partial w}{\partial \eta}=-\bar{\beta}.*w,\text{\,\,\,\,on\,\,\,\,}\Gamma\times[0,\text{\,}T].
     \end{equation}
      Multiplying both sides of equation $(\ref{11})$ by $c_{l}z\in[H^{1}(\Omega)]^{2}$ and integrating over $G_{i,l}$, for $(i,l)\in I$, to obtain
     \begin{equation*}
         c_{l}\int_{G_{i,l}}(\frac{\partial w}{\partial t})^{t}zdx-c_{l}\int_{G_{i,l}}(\bar{\gamma}.*\bar{\Delta}w)^{t}zdx =c_{l}\int_{G_{i,l}}F(w)^{t}zdx.
     \end{equation*}

  Summing this up for $i=1,...,M-1$, and $l=1,...,L$, using the scalar product $\left(\cdot,\cdot\right)_{\cdot}$, given in equation $(\ref{6})$, this gives
    \begin{equation*}
         \left(\frac{\partial w}{\partial t},z\right)_{\cdot}-\left(\bar{\gamma}.*\bar{\Delta}w,z\right)_{\cdot}=\left(F(w),z\right)_{\cdot}
     \end{equation*}

  Applying the integration by parts $(\ref{8})$, simple calculations yield
     \begin{equation}\label{17}
         \left(\frac{\partial w}{\partial t},z\right)_{\cdot}+\left(\bar{\gamma}.*\bar{\nabla}^{t}w,\bar{\nabla}z\right)_{*,\cdot}=\left(F(w),z\right)_{\cdot}-
<\bar{\gamma}.*\bar{\beta}.*w,z>+<\bar{\gamma}.*w,z>_{\cdot},
     \end{equation}
  for every $z\in[H^{1}(\Omega)]^{2}$, where $\left(\cdot,\cdot\right)_{*,\cdot}$, $<\cdot,\cdot>$ and $<\cdot,\cdot>_{\cdot}$, are the bilinear operators given in equation $(\ref{6})$.\\

      The approximation of the vector-valued function $w$ at points $t_{n-\frac{1}{2}}$, $t_{n}$, and $t_{n+\frac{1}{2}}$, using local discrete time steps $\tau_{n}$ gives
      \begin{equation*}
      w(t)=\frac{(t-t_{n})(t-t_{n-\frac{1}{2}})}{\tau_{n}(\tau_{n-\frac{1}{2}}+\tau_{n})}w^{n+\frac{1}{2}}-\frac{(t-t_{n-\frac{1}{2}})
(t-t_{n+\frac{1}{2}})}{\tau_{n}\tau_{n-\frac{1}{2}}}w^{n}+
      \frac{(t-t_{n})(t-t_{n+\frac{1}{2}})}{\tau_{n-\frac{1}{2}}(\tau_{n}+\tau_{n-\frac{1}{2}})}w^{n-\frac{1}{2}}+
       \end{equation*}
       \begin{equation*}
      \frac{1}{6}(t-t_{n-\frac{1}{2}})(t-t_{n})(t-t_{n+\frac{1}{2}})w_{3t}(\theta(t)),
       \end{equation*}
        where $w_{3t}$ denotes $\frac{\partial^{3}w}{\partial t^{3}}$ and $\theta(t)$ is between the maximum and minimum of $t_{n-\frac{1}{2}}$, $t_{n}$, $t_{n+\frac{1}{2}}$ and $t$. The time derivative of this equation provides
     \begin{equation*}
      w_{t}(t)=\frac{2t-t_{n}-t_{n-\frac{1}{2}}}{\tau_{n}(\tau_{n-\frac{1}{2}}+\tau_{n})}w^{n+\frac{1}{2}}-
    \frac{2t-t_{n-\frac{1}{2}}-t_{n+\frac{1}{2}}}{\tau_{n}\tau_{n-\frac{1}{2}}}w^{n}+
      \frac{2t-t_{n}-t_{n+\frac{1}{2}}}{\tau_{n-\frac{1}{2}}(\tau_{n}+\tau_{n-\frac{1}{2}})}w^{n-\frac{1}{2}}+
       \end{equation*}
       \begin{equation*}
      \frac{1}{6}[(t-t_{n-\frac{1}{2}})(t-t_{n})(t-t_{n+\frac{1}{2}})\frac{\partial}{\partial t}(w_{3t}(\theta(t)))+w_{3t}(\theta(t))\frac{d}{dt}((t-t_{n-\frac{1}{2}})(t-t_{n})(t-t_{n+\frac{1}{2}}))].
       \end{equation*}

   So
   \begin{equation}\label{18}
      w_{t}^{n}=\frac{\tau_{n-\frac{1}{2}}}{\tau_{n}(\tau_{n-\frac{1}{2}}+\tau_{n})}w^{n+\frac{1}{2}}+
    \frac{\tau_{n}-\tau_{n-\frac{1}{2}}}{\tau_{n}\tau_{n-\frac{1}{2}}}w^{n}-\frac{\tau_{n}}{\tau_{n-\frac{1}{2}}(\tau_{n}+\tau_{n-\frac{1}{2}})}w^{n-\frac{1}{2}}-
      \frac{\tau_{n}\tau_{n-\frac{1}{2}}}{6}w_{3t}(\theta^{n}).
       \end{equation}

   Applying equation $(\ref{17})$ at the discrete time $t_{n}$, utilizing equation $(\ref{18})$ and rearranging terms to obtain
    \begin{equation*}
         \frac{\tau_{n-\frac{1}{2}}}{\tau_{n}(\tau_{n-\frac{1}{2}}+\tau_{n})}\left(w^{n+\frac{1}{2}},z\right)_{\cdot}+
   \left(\bar{\gamma}.*\bar{\nabla}^{t}w^{n},\bar{\nabla}z\right)_{*,\cdot}=-\frac{\tau_{n}-\tau_{n-\frac{1}{2}}}{\tau_{n}\tau_{n-\frac{1}{2}}}
  \left(w^{n},z\right)_{\cdot}+\frac{\tau_{n}}{\tau_{n-\frac{1}{2}}(\tau_{n-\frac{1}{2}}+\tau_{n})}\left(w^{n-\frac{1}{2}},z\right)_{\cdot}+
     \end{equation*}
   \begin{equation*}
   \left(F(w^{n}),z\right)_{\cdot}-<\bar{\gamma}.*\bar{\beta}.*w^{n},z>+<\bar{\gamma}.*w^{n},z>_{\cdot}+\frac{\tau_{n}\tau_{n-\frac{1}{2}}}{6}
   \left(w_{3t}(\theta^{n}),z\right)_{\cdot}.
     \end{equation*}

   Multiplying both sides of this equation by $\frac{\tau_{n}(\tau_{n}+\tau_{n-\frac{1}{2}})}{\tau_{n-\frac{1}{2}}}$, we obtain
   \begin{equation*}
    \left(w^{n+\frac{1}{2}},z\right)_{\cdot}+\frac{\tau_{n}(\tau_{n}+\tau_{n-\frac{1}{2}})}{\tau_{n-\frac{1}{2}}}
   \left(\bar{\gamma}.*\bar{\nabla}^{t}w^{n},\bar{\nabla}z\right)_{*,\cdot}=-\frac{\tau_{n}^{2}-\tau_{n-\frac{1}{2}}^{2}}{\tau_{n-\frac{1}{2}}^{2}}
  \left(w^{n},z\right)_{\cdot}+\frac{\tau_{n}^{2}}{\tau_{n-\frac{1}{2}}^{2}}\left(w^{n-\frac{1}{2}},z\right)_{\cdot}+
     \end{equation*}
   \begin{equation}\label{19}
   \frac{\tau_{n}(\tau_{n}+\tau_{n-\frac{1}{2}})}{\tau_{n-\frac{1}{2}}}\left[\left(F(w^{n}),z\right)_{\cdot}-<\bar{\gamma}.*\bar{\beta}.*w^{n},z>+
   <\bar{\gamma}.*w^{n},z>_{\cdot}\right]+\frac{\tau_{n}^{2}(\tau_{n}+\tau_{n-\frac{1}{2}})}{6}\left(w_{3t}(\theta^{n}),z\right)_{\cdot}.
     \end{equation}

        Tracking the error term: $\frac{\tau_{n}^{2}(\tau_{n}+\tau_{n-\frac{1}{2}})}{6}\left(w_{3t}(\theta^{n}),z\right)_{\cdot}$, replacing the exact solution $w(t)\in[H^{m}(\Omega)]^{2}\cap[L^{2}_{\sup}(\Omega)]^{2}$ ($m\geq3$), with the approximate one $w_{h}(t)\in \mathcal{W}_{h}$, for $t\in[0,\text{\,}T]$, using equations $(\ref{6})$ and $(\ref{9})$, and rearranging terms, this yields
        \begin{equation*}
    \underset{k=1}{\overset{M_{m}}\sum}\left(w_{h,k}^{n+\frac{1}{2}}.*\rho_{k},z\right)_{\cdot}+\frac{\tau_{n}(\tau_{n}+\tau_{n-\frac{1}{2}})}{\tau_{n-\frac{1}{2}}}
   \underset{k=1}{\overset{M_{m}}\sum}\left(\bar{\gamma}.*\bar{\nabla}^{t}(w_{h,k}^{n}.*\rho_{k}),\bar{\nabla}z\right)_{*,\cdot}=-\frac{\tau_{n}^{2}-\tau_{n-\frac{1}{2}}^{2}}{\tau_{n-\frac{1}{2}}^{2}}
  \underset{k=1}{\overset{M_{m}}\sum}\left(w_{h,k}^{n}.*\rho_{k},z\right)_{\cdot}+
     \end{equation*}
     \begin{equation*}
   \frac{\tau_{n}^{2}}{\tau_{n-\frac{1}{2}}^{2}}\underset{k=1}{\overset{M_{m}}\sum}\left(w_{h,k}^{n-\frac{1}{2}}.*\rho_{k},z\right)_{\cdot}+
   \frac{\tau_{n}(\tau_{n}+\tau_{n-\frac{1}{2}})}{\tau_{n-\frac{1}{2}}}\left[\left(F\left(\underset{k=1}{\overset{M_{m}}\sum}w_{h,k}^{n}.*\rho_{k}\right),z\right)_{\cdot}-
   \underset{k=1}{\overset{M_{m}}\sum} <\bar{\gamma}.*\bar{\beta}.*w_{h,k}^{n}.*\rho_{k},z>+\right.
     \end{equation*}
   \begin{equation}\label{20}
   \left.\underset{k=1}{\overset{M_{m}}\sum}<\bar{\gamma}.*w_{h,k}^{n}.*\rho_{k},z>_{\cdot}\right],
     \end{equation}
   where $w_{h}^{r}$, is expressed in the orthogonal basis $\{\rho_{k},\text{\,}k=1,...,M_{m}\}$, as
   \begin{equation*}
   w_{h}^{r}=\underset{k=1}{\overset{M_{m}}\sum}w_{h,k}^{r}.*\rho_{k},\text{\,\,\,for\,\,\,}r\in\{n-\frac{1}{2},n,n+\frac{1}{2}\}.
     \end{equation*}

  We remind that the coefficients: $w_{h,k}^{r}$, for $k=1,2,...,M_{m}$, are called the "vector components" of $w_{h}^{r}$ in the basis $\{\rho_{k},\text{\,}k=1,...,M_{m}\}$. Equation $(\ref{20})$ denotes the first step of the desired computational technique. It is easy to observe that this equation works with an explicit numerical scheme. As a result, this approximation represents the predictor stage. To construct the second step of the desired algorithm, we should approximate the vector-valued function $w(t)$ at points $t_{n}$, $t_{n+\frac{1}{2}}$,  and $t_{n+1}$. That is,

      \begin{equation*}
      w(t)=\frac{(t-t_{n})(t-t_{n+\frac{1}{2}})}{\tau_{n+\frac{1}{2}}(\tau_{n+\frac{1}{2}}+\tau_{n})}w^{n+1}-\frac{(t-t_{n+1})(t-t_{n})}{\tau_{n}\tau_{n+\frac{1}{2}}}w^{n+\frac{1}{2}}
      +\frac{(t-t_{n+1})(t-t_{n+\frac{1}{2}})}{\tau_{n}(\tau_{n}+\tau_{n+\frac{1}{2}})}w^{n}+
       \end{equation*}
       \begin{equation*}
      \frac{1}{6}(t-t_{n+\frac{1}{2}})(t-t_{n})(t-t_{n+1})w_{3t}(\theta(t)),
       \end{equation*}
        where $\theta(t)$ is between the maximum and minimum of $t_{n}$, $t_{n+\frac{1}{2}}$, $t_{n+1}$ and $t$. Since $\tau_{n+\frac{1}{2}}=\tau_{n}$, the time derivative of this equation at point $t_{n+1}$ results in
     \begin{equation}\label{21}
      w_{t}^{n+1}=\frac{3}{2\tau_{n}}w^{n+1}-\frac{2}{\tau_{n}}w^{n+\frac{1}{2}}+\frac{1}{2\tau_{n}}w^{n}+\frac{\tau_{n}^{2}}{3}w_{3t}(\theta^{n}).
       \end{equation}

   In addition, expanding the Taylor series for $F(w)$, straightforward computations give
     \begin{equation*}
      F(w^{n+1})=2F(w^{n+\frac{1}{2}})-F(w^{n})+2\tau_{n}^{2}[\frac{\partial^{2}(F(w))}{\partial t^{2}}(\theta^{n})+\frac{\partial^{2}(F(w))}{\partial t^{2}}(\theta^{n+\frac{1}{2}})],
       \end{equation*}
   where $t_{n}<\theta^{n}<t_{n+\frac{1}{2}}$ and $t_{n+\frac{1}{2}}<\theta^{n+\frac{1}{2}}<t_{n+1}$. Setting
   \begin{equation}\label{22b}
    G^{n}=2[\frac{\partial^{2}(F(w))}{\partial t^{2}}(\theta^{n})+\frac{\partial^{2}(F(w))}{\partial t^{2}}(\theta^{n+\frac{1}{2}})],
     \end{equation}
    this approximation becomes
    \begin{equation}\label{22a}
      F(w^{n+1})=2F(w^{n+\frac{1}{2}})-F(w^{n})+\tau_{n}^{2}G^{n}.
     \end{equation}

    Since $F(w)\in H^{4}(0,T;\text{\,}[H^{m}(\Omega)]^{2}\cap[L^{2}_{\sup}(\Omega)]^{2})$, using equation $(\ref{22b})$ it is easy to see that\\ $G\in H^{2}(0,T;\text{\,}[H^{m}(\Omega)]^{2}\cap[L^{2}_{\sup}(\Omega)]^{2})$. Combining equations $(\ref{21})$, $(\ref{22a})$ and $(\ref{17})$, it is not hard to observe that
    \begin{equation*}
    \left(w^{n+1},z\right)_{\cdot}+\frac{2\tau_{n}}{3}\left(\bar{\gamma}.*\bar{\nabla}^{t}w^{n+1},\bar{\nabla}z\right)_{*,\cdot}=\frac{4}{3}
  \left(w^{n+\frac{1}{2}},z\right)_{\cdot}-\frac{1}{3}\left(w^{n},z\right)_{\cdot}+\frac{2\tau_{n}}{3}\left[\left(2F(w^{n+\frac{1}{2}})-F(w^{n}),z\right)_{\cdot}-\right.
     \end{equation*}
   \begin{equation}\label{22}
  \left.<\bar{\gamma}.*\bar{\beta}.*w^{n+1},z>+<\bar{\gamma}.*w^{n+1},z>_{\cdot}\right]+\frac{2\tau_{n}^{3}}{9}\left(3G^{n}-w_{3t}(\theta^{n+1}),z\right)_{\cdot},
     \end{equation}
   where $G^{n}$ is given by equation $(\ref{22b})$. Truncating the error term: $\frac{2\tau_{n}^{3}}{9}\left(3G^{n}-w_{3t}(\theta^{n+1}),z\right)_{\cdot}$, replacing the analytical solution $w(t)\in[H^{m}(\Omega)]^{2}\cap[L^{2}_{\sup}(\Omega)]^{2}$, with the computed one $w_{h}(t)\in \mathcal{W}_{h}$, for all $0\leq t\leq T$, utilizing both equations $(\ref{6})$ and $(\ref{9})$, and rearranging terms, result in
        \begin{equation*}
    \underset{k=1}{\overset{M_{m}}\sum}\left(w_{h,k}^{n+1}.*\rho_{k},z\right)_{\cdot}+\frac{2\tau_{n}}{3}
   \underset{k=1}{\overset{M_{m}}\sum}\left(\bar{\gamma}.*\bar{\nabla}^{t}(w_{h,k}^{n+1}.*\rho_{k}),\bar{\nabla}z\right)_{*,\cdot}=\frac{1}{3}\underset{k=1}{\overset{M_{m}}\sum}
  \left[4\left(w_{h,k}^{n+\frac{1}{2}}.*\rho_{k},z\right)_{\cdot}-\left(w_{h,k}^{n}.*\rho_{k},z\right)_{\cdot}\right]+
     \end{equation*}
     \begin{equation*}
     \frac{2\tau_{n}}{3}\left[\left(2F\left(\underset{k=1}{\overset{M_{m}}\sum}w_{h,k}^{n+\frac{1}{2}}.*\rho_{k}\right)-
   F\left(\underset{k=1}{\overset{M_{m}}\sum}w_{h,k}^{n}.*\rho_{k}\right),z\right)_{\cdot}-
   \underset{k=1}{\overset{M_{m}}\sum} <\bar{\gamma}.*\bar{\beta}.*w_{h,k}^{n+1}.*\rho_{k},z>+\right.
     \end{equation*}
   \begin{equation}\label{23}
   \left.\underset{k=1}{\overset{M_{m}}\sum}<\bar{\gamma}.*w_{h,k}^{n+1}.*\rho_{k},z>_{\cdot}\right].
     \end{equation}

   Equation $(\ref{23})$ deals with an implicit method, thus it denotes the second stage of the constructed predictor-corrector approach with orthogonal spline collocation FEM for simulating the FitzHugh-Nagumo model $(\ref{1})$-$(\ref{3})$. It is worth mentioning that the errors increased at the predictor stage $(\ref{20})$ are balanced by the ones decreased at the corrector step $(\ref{23})$ so that the stability of the new strategy is maintained. Additionally, the use of variable time steps greatly reduces the numerical oscillations (due to the "convection terms": $\left(\bar{\gamma}.*\bar{\nabla}^{t}(w_{h,k}^{n}.*\rho_{k}),\bar{\nabla}z\right)_{*,\cdot}$, $<\bar{\gamma}.*w_{h,k}^{n}.*\rho_{k},z>_{\cdot}$ and $<\bar{\gamma}.*\bar{\beta}.*w_{h,k}^{n}.*\rho_{k},z>$), while the orthogonal spline collocation FEM extracts and retains the essential characteristics of the approximate solution within a finite-dimensional vector space, hence significantly reducing the computational costs. More precisely, the approximations given by the orthogonal spline collocation finite element technique deal with both the solution and its space derivatives throughout a computational domain. As a result, it achieves spatial high-accuracy. Regarding the advantages of such methods, we refer the readers to \cite{28wcqh,25wcqh,27wcqh,34wcqh} and references therein. Finally, the linearization of the nonlinear reaction term (that is, $F(w^{n+1})$) overcomes the difficulties to compute the corrected solution provided at the corrector stage. In fact, with this linearization, the corrector stage should yield a block system (or two simple systems) of linear equations which are easily solved by calculating the inverse of the coefficients matrix. This greatly reduces the time consuming of the proposed algorithm.\\

  To start the numerical scheme defined by equations $(\ref{20})$ and $(\ref{23})$, both initial data $w_{h,k}^{0}$ and $w_{h,k}^{\frac{1}{2}}$, for $k=1,2,...,M_{m}$, are needed. But the terms $w_{h,k}^{0}$ can be directly determined from the initial condition $w_{0}$ given by equation $(\ref{12})$ as
   \begin{equation}\label{24}
   w_{h,k}^{0}=(P_{h}w_{0})_{k}, \text{\,\,\,for\,\,\,\,} k=1,2,...,M_{m},
     \end{equation}
   where $P_{h}w_{0}=\underset{k=1}{\overset{M_{m}}\sum}(P_{h}w_{0})_{k}.*\rho_{k}$, and $P_{h}$ is the $L^{2}(\Omega)\times L^{2}(\Omega)$-projection given by equation $(\ref{10})$. Utilizing the Taylor series expansion and equation $(\ref{11})$, simple computations provide
     \begin{equation}\label{25}
   w^{\frac{1}{2}}=w_{0}+\frac{\tau_{0}}{2}[F(w_{0})+F(w^{\frac{1}{2}})+\bar{\gamma}.*\bar{\Delta}(w_{0}+w^{\frac{1}{2}})]+\frac{\tau_{0}^{2}}{4}(w_{2t}(\theta_{0})-
   w_{2t}(\theta^{\frac{1}{2}})),
     \end{equation}
   where $\theta_{0},\theta^{\frac{1}{2}}\in(0,t_{\frac{1}{2}})$. Omitting the error term: $\frac{\tau_{0}^{2}}{4}(w_{2t}(\theta_{0})-w_{2t}(\theta^{\frac{1}{2}}))$, this equation should be approximated as
   \begin{equation}\label{26}
   \widetilde{w}^{\frac{1}{2}}=w_{0}+\frac{\tau_{0}}{2}[F(w_{0})+F(w^{\frac{1}{2}})+\bar{\gamma}.*\bar{\Delta}(w_{0}+w^{\frac{1}{2}})].
     \end{equation}

   Takes
   \begin{equation}\label{27}
   w_{h,k}^{\frac{1}{2}}=(P_{h}\widetilde{w}^{\frac{1}{2}})_{k}, \text{\,\,\,for\,\,\,\,} k=1,2,...,M_{m},
     \end{equation}
   where $(P_{h}\widetilde{w}^{\frac{1}{2}})_{k}$, $k=1,...,M_{m}$, designate the components of $P_{h}\widetilde{w}^{\frac{1}{2}}$, in the orthogonal basis $\{\rho_{k},\text{\,}k=1,...,M_{m}\}$. Plugging equations $(\ref{20})$, $(\ref{23})$, $(\ref{24})$ and $(\ref{27})$, to get the new predictor-corrector approach with orthogonal spline collocation FEM for solving a $d$-dimensional system of FitzHugh-Nagumo equations $(\ref{1})$-$(\ref{3})$, that is, for $n=1,2,...,N-1$,
    \begin{equation*}
    \underset{k=1}{\overset{M_{m}}\sum}\left(w_{h,k}^{n+\frac{1}{2}}.*\rho_{k},z\right)_{\cdot}+\frac{\tau_{n}(\tau_{n}+\tau_{n-\frac{1}{2}})}{\tau_{n-\frac{1}{2}}}
   \underset{k=1}{\overset{M_{m}}\sum}\left(\bar{\gamma}.*\bar{\nabla}^{t}(w_{h,k}^{n}.*\rho_{k}),\bar{\nabla}z\right)_{*,\cdot}=-\frac{\tau_{n}^{2}-\tau_{n-\frac{1}{2}}^{2}}{\tau_{n-\frac{1}{2}}^{2}}
  \underset{k=1}{\overset{M_{m}}\sum}\left(w_{h,k}^{n}.*\rho_{k},z\right)_{\cdot}+
     \end{equation*}
     \begin{equation*}
   \frac{\tau_{n}^{2}}{\tau_{n-\frac{1}{2}}^{2}}\underset{k=1}{\overset{M_{m}}\sum}\left(w_{h,k}^{n-\frac{1}{2}}.*\rho_{k},z\right)_{\cdot}+
   \frac{\tau_{n}(\tau_{n}+\tau_{n-\frac{1}{2}})}{\tau_{n-\frac{1}{2}}}\left[\left(F\left(\underset{k=1}{\overset{M_{m}}\sum}w_{h,k}^{n}.*\rho_{k}\right),z\right)_{\cdot}-
   \underset{k=1}{\overset{M_{m}}\sum} <\bar{\gamma}.*\bar{\beta}.*w_{h,k}^{n}.*\rho_{k},z>+\right.
     \end{equation*}
   \begin{equation}\label{s1}
   \left.\underset{k=1}{\overset{M_{m}}\sum}<\bar{\gamma}.*w_{h,k}^{n}.*\rho_{k},z>_{\cdot}\right],\text{\,\,\,for\,\,\,}\forall z\in[H^{1}(\Omega)]^{2},
     \end{equation}
    \begin{equation*}
    \underset{k=1}{\overset{M_{m}}\sum}\left(w_{h,k}^{n+1}.*\rho_{k},z\right)_{\cdot}+\frac{2\tau_{n}}{3}
   \underset{k=1}{\overset{M_{m}}\sum}\left(\bar{\gamma}.*\bar{\nabla}^{t}(w_{h,k}^{n+1}.*\rho_{k}),\bar{\nabla}z\right)_{*,\cdot}=\frac{1}{3}\underset{k=1}{\overset{M_{m}}\sum}
  \left[4\left(w_{h,k}^{n+\frac{1}{2}}.*\rho_{k},z\right)_{\cdot}-\left(w_{h,k}^{n}.*\rho_{k},z\right)_{\cdot}\right]+
     \end{equation*}
     \begin{equation*}
     \frac{2\tau_{n}}{3}\left[\left(2F\left(\underset{k=1}{\overset{M_{m}}\sum}w_{h,k}^{n+\frac{1}{2}}.*\rho_{k}\right)-
   F\left(\underset{k=1}{\overset{M_{m}}\sum}w_{h,k}^{n}.*\rho_{k}\right),z\right)_{\cdot}-
   \underset{k=1}{\overset{M_{m}}\sum} <\bar{\gamma}.*\bar{\beta}.*w_{h,k}^{n+1}.*\rho_{k},z>+\right.
     \end{equation*}
   \begin{equation}\label{s2}
   \left.\underset{k=1}{\overset{M_{m}}\sum}<\bar{\gamma}.*w_{h,k}^{n+1}.*\rho_{k},z>_{\cdot}\right],\text{\,\,\,for\,\,\,}\forall z\in[H^{1}(\Omega)]^{2},
     \end{equation}
   subject to initial conditions
    \begin{equation}\label{s3}
   w_{h,k}^{0}=(P_{h}w_{0})_{k},\text{\,\,\,\,\,}w_{h,k}^{\frac{1}{2}}=(P_{h}\widetilde{w}^{\frac{1}{2}})_{k}, \text{\,\,\,for\,\,\,\,} k=1,2,...,M_{m},
     \end{equation}
   and "local Neumann boundary conditions"
   \begin{equation}\label{s4}
    \underset{k=1}{\overset{M_{m}}\sum}[[w_{h,k}^{n}.*\bar{\nabla}^{t}(\rho_{k})]]=0,\text{\,on\,}\Gamma_{i},
    \text{\,if\,}\Gamma_{i}\cap\Gamma=\emptyset,\text{\,}\underset{k=1}{\overset{M_{m}}\sum}[[w_{h,k}^{n}.*\bar{\nabla}^{t}(\rho_{k})]]=
   -\underset{k=1}{\overset{M_{m}}\sum}\bar{\beta}.*w_{h,k}^{n}.*\rho_{k},\text{\,on\,}\Gamma_{i}^{b}=\Gamma_{i}\cap\Gamma\neq\emptyset,
     \end{equation}
   for $n=0,1,...,N,$ and $i=1,2,...,M$.

     \section{Stability analysis and error estimates}\label{sec3}
     This Section analyzes the stability and error estimates of the developed computational technique $(\ref{s1})$-$(\ref{s4})$ for solving the initial-boundary value problem $(\ref{11})$-$(\ref{13})$.\\

    Since $q_{h}$ is the $L^{2}$-projection from $L^{2}(\Omega)$ onto $\mathcal{U}^{(k)}$, for $k=1,2$, it is well known in the literature \cite{1en}, that $q_{h}$ and $P_{h}$ satisfy: for every $Q_{i}\in\mathcal{F}_{h}$,
    \begin{equation}\label{28}
     \nabla_{d}(q_{h}u)=P_{h}(\nabla u),\text{\,\,\,\,}\forall u\in H^{1}(\Omega),
        \end{equation}
   where $P_{h}$ is the orthogonal projection defined in equation $(\ref{6})$, and $\nabla_{d}$ means the weak gradient defined as: for any $u_{h}\in\mathcal{U}^{(k)}$
   \begin{equation}\label{29}
     \int_{Q_{i}}(\nabla_{d}u_{h})^{t}vdQ_{i}=\int_{\Gamma_{i}}u_{h}v^{t}n_{\Gamma_{i}}d\Gamma_{i}-\int_{Q_{i}}u_{h}\nabla\cdot vdQ_{i},\text{\,\,\,\,}\forall v\in[H^{1}(\Omega)]^{2},
        \end{equation}
   where $\Gamma_{i}$ is the boundary of $Q_{i}$ and $n_{\Gamma_{i}}$ represents the outward unit normal vector on $\Gamma_{i}$.

   \begin{lemma}\label{l1}
      For every $Q_{i}\in\mathcal{F}_{h}$, and any $U\in[H^{1}(\Omega)]^{2}$, the orthogonal projections $\mathbb{P}_{h}$ and $P_{h}$ defined in equation $(\ref{10})$, satisfy
        \begin{equation*}
      \bar{\nabla}_{r}(P_{h}U)=\mathbb{P}_{h}(\bar{\nabla} U),
        \end{equation*}
        where $\bar{\nabla}_{r}=[\nabla_{r},\nabla_{r}]$, denotes the restriction of the operator $\bar{\nabla}$ (defined by equation $(\ref{6})$) on $Q_{i}$, for $i=1,2,...,M$, that is, for all $U\in[H^{1}(\Omega)]^{2}$, $\bar{\nabla}U|_{Q_{i}}=\bar{\nabla}_{r}U$. Moreover, $\bar{\nabla}_{r}$ represents the discrete gradient defined on $\mathcal{W}_{h}$.
        \end{lemma}

       \begin{proof}
        Firstly, using equation $(\ref{29})$ and the integration by parts, it is not hard to show that $\nabla_{d}u=\nabla_{r}u$, for every $u\in H^{1}(Q_{i})$. Further,
        setting $U=(u_{1},u_{2})^{t}\in[H^{1}(\Omega)]^{2}$, we have $\bar{\nabla}_{r}U=[\nabla_{r}u_{1},\nabla_{r}u_{2}]$. Utilizing equation $(\ref{28})$, it holds
        \begin{equation*}
         \bar{\nabla}_{r}(P_{h}U)=(\nabla_{r}(q_{h}u_{1}),\nabla_{r}(q_{h}u_{2}))=(\nabla_{d}(q_{h}u_{1}),\nabla_{d}(q_{h}u_{2}))=(P_{h}\nabla u_{1},P_{h}\nabla u_{2})=
       \mathbb{P}_{h}(\nabla u_{1},\nabla u_{2})=\mathbb{P}_{h}(\bar{\nabla}U).
       \end{equation*}
       This ends the proof of Lemma $\ref{l1}$.
        \end{proof}

     \begin{lemma}\label{l2}\cite{10nb}
     For every $v\in H^{m}(\Omega)$ (where $m\geq3$ is an integer), the following estimates hold
    \begin{equation}\label{30}
     \underset{i=1}{\overset{M}\sum}\left[\|v-q_{h}v\|_{Q_{i}}^{2}+h^{2}_{Q_{i}}\|\nabla(v-q_{h}v)\|_{Q_{i}}^{2}\right]\leq C_{0}h^{2m}\|v\|_{m}^{2},\text{\,\,and\,\,\,\,}
     \underset{i=1}{\overset{M}\sum}\|\nabla v-P_{h}(\nabla v)\|_{Q_{i}}^{2}\leq C_{0}h^{2(m-1)}\|v\|_{m}^{2},
    \end{equation}
     where $h_{Q_{i}}$ means the diameter of $Q_{i}$, $\|v\|_{Q_{i}}=\sqrt{\int_{Qi}v^{t}vdx}$, $\|\cdot\|_{m}$ denotes the $H^{m}$-norm, and $C_{0}$ is a positive constant independent of the grid size $h$.
     \end{lemma}

     \begin{lemma}\label{l3}
     For every $U=(u_{1},u_{2})^{t}\in[H^{m}(\Omega)]^{2}$, the given inequalities are satisfied
    \begin{equation}\label{31}
     \|U-P_{h}U\|_{\cdot}^{2}\leq \hat{\beta}_{0}h^{2m}\|U\|_{[H^{m}(\Omega)]^{2}}^{2},\text{\,\,and\,\,\,}
     \|\bar{\nabla}U-\mathbb{P}_{h}(\bar{\nabla}U)\|_{*,\cdot}^{2}\leq \hat{\beta}_{0}h^{2m-2}\|U\|_{[H^{m}(\Omega)]^{2}}^{2},
    \end{equation}
     where $\hat{\beta}_{0}=C_{0}(p+2)L\underset{1\leq l\leq L}{\max}{\{c_{l}\}}$, $C_{0}$ is the positive constant given in estimate $(\ref{30})$, and $p\geq0$ designates
     the constant integer defined in the assumptions satisfied by the triangulation $\mathcal{F}_{h}$.
     \end{lemma}

       \begin{proof}
      It follows from equations $(\ref{6})$ and $(\ref{7})$ that
     \begin{equation}\label{32}
      \|U-P_{h}U\|_{\cdot}^{2}=\left(U-P_{h}U,U-P_{h}U\right)_{.}=\underset{(i,l)\in I}{\sum}c_{l}\int_{G_{i,l}}(U-P_{h}U)^{t}(U-P_{h}U)dx.
     \end{equation}

     Since $G_{i,l}=(1-\alpha_{l})Q_{i}+\alpha_{l}Q_{i+1}$, where $0\leq\alpha_{l}\leq1$, $Q_{i}$ and $Q_{i+1}$ are consecutive triangles/tetrahedra, so $Q_{i}\cap Q_{i+1}=e$, where $e$ represents the common edge or face. So, it follows from assumption $(b)$ on the triangulation $\mathcal{F}_{h}$ that there are elements $\bar{Q}_{i_{k}}\in\mathcal{F}_{h}$, for $k=1,...,p$, so that $Q_{i}\cap Q_{i+1}\cap(\underset{k=1}{\overset{p}\bigcap}\bar{Q}_{i_{k}})=e$ and $Q_{i}\cup Q_{i+1}\cup(\underset{k=1}{\overset{p}\bigcup}\bar{Q}_{i_{k}})$ is a convex set. Thus, $(1-\alpha_{l})Q_{i}+\alpha_{l}Q_{i+1}\subset
  Q_{i}\cup Q_{i+1}\cup(\underset{k=1}{\overset{p}\bigcup}\bar{Q}_{i_{k}})$, for $i=1,...,M-1$, and $l=1,...,L$. Utilizing this fact, equation $(\ref{32})$ implies
    \begin{equation*}
      \|U-P_{h}U\|_{\cdot}^{2}\leq\underset{i=1}{\overset{M-1}\sum}\underset{l=1}{\overset{L}\sum}c_{l}\left(\underset{k=0}{\overset{1}\sum}\int_{Q_{i+k}}(U-P_{h}U)^{t}(U-P_{h}U)dx+
     \underset{k=1}{\overset{p}\sum}\int_{\bar{Q}_{i_{k}}}(U-P_{h}U)^{t}(U-P_{h}U)dx\right)=
     \end{equation*}
     \begin{equation*}
     \underset{l=1}{\overset{L}\sum}c_{l}\underset{i=1}{\overset{M-1}\sum}\underset{j=1}{\overset{2}\sum}\left[\underset{k=0}{\overset{1}\sum}
      \int_{Q_{i+k}}(u_{j}-q_{h}u_{j})^{2}dx+\underset{k=1}{\overset{p}\sum}\int_{\bar{Q}_{i_{k}}}(u_{j}-q_{h}u_{j})^{2}dx\right]\leq
     \end{equation*}
     \begin{equation*}
      (p+2)L\underset{1\leq l\leq L}{\max}{\{c_{l}\}}\underset{i=1}{\overset{M}\sum}\int_{Q_{i}}[(u_{1}-q_{h}u_{1})^{2}+(u_{2}-q_{h}u_{2})^{2}]dx.
     \end{equation*}

      This ends the proof of the first estimate in relation $(\ref{31})$ thanks to the first inequality in relation $(\ref{30})$.
     \begin{equation*}
     \|\bar{\nabla}U-\mathbb{P}_{h}(\bar{\nabla}U)\|_{*,.}^{2}=\underset{i=1}{\overset{M-1}\sum}\underset{l=1}{\overset{L}\sum}\underset{j=1}{\overset{2}\sum}
      c_{l}\int_{G_{i,l}}(\nabla u_{j}-P_{h}(\nabla u_{j}))^{t}(\nabla u_{j}-P_{h}(\nabla u_{j}))dx \leq \underset{i=1}{\overset{M-1}\sum}\underset{l=1}{\overset{L}\sum}\underset{j=1}{\overset{2}\sum}c_{l}\left(\underset{k=0}{\overset{1}\sum}\right.
      \end{equation*}
      \begin{equation*}
      \left.\int_{Q_{i+k}}(\nabla u_{j}-P_{h}(\nabla u_{j}))^{t}(\nabla u_{j}-P_{h}(\nabla u_{j}))dx+\underset{k=1}{\overset{p}\sum}\int_{\bar{Q}_{i_{k}}}
      (\nabla u_{j}-P_{h}(\nabla u_{j}))^{t}(\nabla u_{j}-P_{h}(\nabla u_{j}))dx\right)\leq (p+2)L\times
     \end{equation*}
     \begin{equation*}
      \underset{1\leq l\leq L}{\max}{\{c_{l}\}}\underset{i=1}{\overset{M}\sum}\underset{j=1}{\overset{2}\sum}\int_{Q_{i}}(\nabla u_{j}-P_{h}
     (\nabla u_{j}))^{t}(\nabla u_{j}-P_{h}(\nabla u_{j}))dx=(p+2)L\underset{1\leq l\leq L}{\max}{\{c_{l}\}}\underset{i=1}{\overset{M}\sum}\underset{j=1}{\overset{2}\sum}
     \|\nabla u_{j}-P_{h}(\nabla u_{j})\|_{Q_{i}}^{2},
     \end{equation*}
     where $"\times"$ designates the usual multiplication in $\mathbb{R}$. The proof of the second estimate in equation $(\ref{31})$ is ended thanks to the second inequality in relation $(\ref{30})$. This completes the proof of Lemma $\ref{l3}$.
     \end{proof}

     \begin{corollary}
     For every $U=(u_{1},u_{2})^{t}\in[H^{m}(\Omega)]^{2}$, it holds
    \begin{equation*}
     \|\bar{\nabla}(U-P_{h}U)\|_{*,\cdot}^{2}\leq C_{0}(p+2)L\underset{1\leq l\leq L}{\max}{\{c_{l}\}}h^{2(m-1)}\|U\|_{[H^{m}(\Omega)]^{2}}^{2}.
    \end{equation*}
     \end{corollary}
     \begin{proof}
      Plugging Lemmas $\ref{l1}$ and $\ref{l3}$ to get the result.
     \end{proof}

     \begin{lemma}\label{l4}
      Consider $U=(u_{1},u_{2})^{t}$ and $V=(v_{1},v_{2})^{t}$ be two elements in $[L^{2}_{\sup}(\Omega)]^{2}$, so the vector-valued function $F$ satisfies
      \begin{equation*}
      \|F(U)-F(V)\|_{\cdot}\leq C_{F}\|U-V\|_{\cdot},
       \end{equation*}
     where $C_{F}^{2}=2[1+\max\{\theta_{1}^{2},\theta_{2}^{2}\}+10L(p+2)\underset{1\leq l\leq L}{\max}{\{c_{l}\}}C^{2}_{\sup}[L(p+2)\underset{1\leq l\leq L}{\max}{\{c_{l}\}}C_{\sup}^{2}+\frac{3}{2}+2(1+\theta_{3})^{2}]]$.
     \end{lemma}

      \begin{proof}
       First, $F(U)=[f_{1}(u_{1})-g(u_{1},u_{2}),f_{2}(u_{1},u_{2})]^{t}$ and $F(V)=[f_{1}(v_{1})-g(v_{1},v_{2}),f_{2}(v_{1},v_{2})]^{t}$. Utilizing this it holds
       \begin{equation*}
      \|F(U)-F(V)\|_{\cdot}^{2}=\|f_{1}(u_{1})-f_{1}(v_{1})-(g(u_{1},u_{2})-g(v_{1},v_{2}))\|_{.}^{2}+\|f_{2}(u_{1},u_{2})-f_{2}(v_{1},v_{2})\|_{.}^{2}\leq
       2[\|f_{1}(u_{1})-f_{1}(v_{1})\|_{.}^{2}+
       \end{equation*}
      \begin{equation*}
      \|g(u_{1},u_{2})-g(v_{1},v_{2})\|^{2}]+\|f_{2}(u_{1},u_{2})-f_{2}(v_{1},v_{2})\|_{.}^{2}=2[\|u_{1}(u_{1}-1)(\theta_{3}-u_{1})-v_{1}(v_{1}-1)(\theta_{3}-v_{1})\|_{.}^{2}
        +\|u_{2}-v_{2}\|_{.}^{2}]+
       \end{equation*}
       \begin{equation*}
      \|\theta_{1}u_{1}-\theta_{2}u_{2}-(\theta_{1}v_{1}-\theta_{2}v_{2})\|_{.}^{2}=2[\|u_{1}^{3}-v_{1}^{3}-(1+\theta_{3})(u_{1}^{2}-v_{1}^{2})+\theta_{3}(u_{1}-v_{1})\|_{.}^{2}+
      \|u_{2}-v_{2}\|_{.}^{2}]+\|\theta_{1}(u_{1}-v_{1})-\theta_{2}(u_{2}-v_{2})\|_{.}^{2}
       \end{equation*}
      \begin{equation}\label{32a}
      \leq 2[\|u_{1}-v_{1}\|_{.}^{2}\|u_{1}^{2}+u_{1}v_{1}+v_{1}^{2}-(1+\theta_{3})(u_{1}+v_{1})+\theta_{3}\|_{.}^{2}+
      \|u_{2}-v_{2}\|_{.}^{2}+\max\{\theta_{1}^{2},\theta_{2}^{2}\}(\|u_{1}-v_{1}\|_{.}^{2}+\|u_{2}-v_{2}\|_{.}^{2})].
       \end{equation}

       Since, $u_{1},v_{1}\in L_{\sup}^{2}(\Omega)$, performing simple computations, it holds
      \begin{equation*}
      \|u_{1}^{2}+u_{1}v_{1}+v_{1}^{2}-(1+\theta_{3})(u_{1}+v_{1})+\theta_{3}\|_{.}^{2}\leq 10L(p+2)\underset{1\leq l\leq L}{\max}{\{c_{l}\}}C^{2}_{\sup}[L(p+2)\underset{1\leq l\leq L}{\max}{\{c_{l}\}}C_{\sup}^{2}+\frac{3}{2}+2(1+\theta_{3})^{2}],
       \end{equation*}
  where $p$ is the positive integer defined in the assumptions on the triangulation $\mathcal{F}_{h}$. Using this fact, it is not hard to observe that inequality $(\ref{32a})$ implies
      \begin{equation*}
        \|F(U)-F(V)\|_{\cdot}^{2}\leq C_{F}^{2}\|U-V\|_{\cdot}^{2},
       \end{equation*}
     where $C_{F}=\sqrt{2[1+\max\{\theta_{1}^{2},\theta_{2}^{2}\}+10L(p+2)\underset{1\leq l\leq L}{\max}{\{c_{l}\}}C^{2}_{\sup}[L(p+2)\underset{1\leq l\leq L}{\max}{\{c_{l}\}}C_{\sup}^{2}+\frac{3}{2}+2(1+\theta_{3})^{2}]]}$.
     This ends the proof of Lemma $\ref{l4}$.
      \end{proof}

     \begin{lemma}\label{l5}
     Suppose that $U=(u_{1},u_{2})^{t},V=(v_{1},v_{2})^{t}\in [H^{2}(\Omega)]^{2}$, then the following estimates are satisfied
    \begin{equation*}
     <\bar{\gamma}.*U,V>_{.}=0,\text{\,\,and\,\,\,}|<U,V>|\leq \hat{C}_{1}(p+2)\|V\|_{.}^{\frac{1}{2}}
     \|\bar{\nabla}V\|_{*,.}^{\frac{1}{2}}\|\bar{\beta}.*U\|_{.}^{\frac{1}{2}}\|\bar{\nabla}(\bar{\beta}.*U)\|_{*,.}^{\frac{1}{2}},
    \end{equation*}
     where $\hat{C}_{1}$, is a positive constant independent of the space size $h$.
     \end{lemma}

    \begin{proof}
     It follows from equations $(\ref{6})$ that
     \begin{equation*}
    |<\bar{\gamma}.*U,V>_{\cdot}|=\left|\underset{(i,l)\in I\setminus I_{\Gamma}}{\sum}c_{l}\int_{\Gamma_{i,l}}V^{t}[[\bar{\nabla}^{t}
    (\bar{\gamma}.*U)]]d\Gamma_{i,l}\right|\leq \underset{(i,l)\in I\setminus I_{\Gamma}}{\sum}c_{l}\int_{\Gamma_{i,l}}|V^{t}[[\bar{\nabla}^{t}
    (\bar{\gamma}.*U)]]|d\Gamma_{i,l}
     \end{equation*}

     But $G_{i,l}\subset Q_{i}\cup Q_{i+1}\cup(\underset{k=1}{\overset{p}\bigcup}\bar{Q}_{i_{k}})$, $G_{i,l}\cap\Gamma=\emptyset$ and $\Gamma_{i,l}\cap\Gamma=\emptyset$,
     for $(i,l)\in I\setminus I_{\Gamma}$, imply $G_{i,l}\subset(Q_{i}\cup Q_{i+1}\cup(\underset{k=1}{\overset{p}\bigcup}\bar{Q}_{i_{k}}))\cap\Gamma^{c}$, and
     $\Gamma_{i,l}\subset\Gamma^{c}$, for $(i,l)\in I\setminus I_{\Gamma}$. This implies $\Gamma_{i,l}\subset(\Gamma_{i}\cup \Gamma_{i+1}\cup(\underset{k=1}{\overset{p}\bigcup}\bar{\Gamma}_{i_{k}}))\cap\Gamma^{c}$, where the inclusion follows from the assumptions provided below equation  $(\ref{16})$), $\Gamma^{c}$ means the set complement of $\Gamma$ in $\Omega$ and $\bar{\Gamma}_{i_{k}}$ denotes the boundary of $\bar{Q}_{i_{k}}$. This indicates that
   $\Gamma_{i,l}\subset(\Gamma_{i}\cap\Gamma^{c})\cup(\Gamma_{i+1}\cap\Gamma^{c})\cup(\underset{k=1}{\overset{p}\bigcup}(\bar{\Gamma}_{i_{k}}\cap\Gamma^{c}))$, for any $(i,l)\in I\setminus I_{\Gamma}$. Thus,
   \begin{equation*}
    |<\bar{\gamma}.*U,V>_{\cdot}|\leq\underset{(i,l)\in I\setminus I_{\Gamma}}{\sum}c_{l}\left(\underset{k=0}{\overset{1}\sum}\int_{\Gamma_{i+k}\cap\Gamma^{c}}
|\{\{V\}\}^{t}[[\bar{\nabla}^{t}(\bar{\gamma}.*U)]]|d\Gamma_{i} +\underset{k=1}{\overset{p}\sum}\int_{\bar{\Gamma}_{i_{k}}\cap\Gamma^{c}}
|\{\{V\}\}^{t}[[\bar{\nabla}^{t}(\bar{\gamma}.*U)]]|d\Gamma_{i}\right).
     \end{equation*}

   Since $U\in [H^{2}(\Omega)]^{2}$, so $\bar{\nabla}^{t}U$ is continuous on $\Omega$. Thus, $[[\bar{\nabla}^{t}U]]=0$, on $\Gamma_{i}$, with $\Gamma_{i}\cap\Gamma=\emptyset$. Simple calculations give $[[\bar{\nabla}^{t}(\bar{\gamma}.*U)]]=\bar{\gamma}.*[[\bar{\nabla}^{t}U]]=0$, on $\Gamma_{i}$, for $(i,l)\in I\setminus I_{\Gamma}$. This implies $[[\bar{\nabla}^{t}(\bar{\gamma}.*U)]]=0$, on $\Gamma_{i+k}\cap\Gamma^{c}$ and $\bar{\Gamma}_{i_{k}}\cap\Gamma^{c}$, for every $(i,l)\in I\setminus I_{\Gamma}$, and $k=1,...,p$. This ends the proof of the first equality in Lemma $\ref{l5}$. Now, we should establish the second estimate.

    \begin{equation*}
    |<U,V>|=\left|\underset{(i,l)\in I_{\Gamma}}{\sum}c_{l}\int_{\Gamma_{i,l}^{b}}V^{t}(\bar{\nabla}^{t}U)n_{\Gamma_{i,l}^{b}}d\Gamma\right|\leq \underset{(i,l)\in I_{\Gamma}}{\sum}c_{l}\int_{\Gamma_{i,l}^{b}}|V^{t}(\bar{\nabla}^{t}U)n_{\Gamma_{i,l}^{b}}|d\Gamma.
     \end{equation*}

     But $G_{i,l}\subset Q_{i}\cup Q_{i+1}\cup(\underset{k=1}{\overset{p}\bigcup}\bar{Q}_{i_{k}})$, implies $\Gamma_{i,l}^{b}\subset\Gamma_{i}\cup \Gamma_{i+1}\cup(\underset{k=1}{\overset{p}\bigcup}\bar{\Gamma}_{i_{k}})$, for all $(i,l)\in I_{\Gamma}$ (according to the assumption below equation $(\ref{16})$). So, $\Gamma_{i,l}^{b}=\Gamma_{i,l}\cap\Gamma\subset(\Gamma_{i}\cap\Gamma)\cup(\Gamma_{i+1}\cap\Gamma)\cup(\underset{k=1}{\overset{p}\bigcup}(\bar{\Gamma}_{i_{k}}\cap\Gamma))$, for any $(i,l)\in I_{\Gamma}$. Thus,
   \begin{equation*}
    |<U,V>|\leq\underset{(i,l)\in I_{\Gamma}}{\sum}c_{l}\left(\underset{k=0}{\overset{1}\sum}\int_{\Gamma_{i+k}\cap\Gamma}
|V^{t}(\bar{\nabla}^{t}U)n_{\Gamma_{i}^{b}}|d\Gamma +\underset{k=1}{\overset{p}\sum}\int_{\bar{\Gamma}_{i_{k}}\cap\Gamma}|V^{t}(\bar{\nabla}^{t}
    U)n_{\Gamma_{i}^{b}}|d\Gamma\right).
     \end{equation*}

   But $(\bar{\nabla}^{t}U)n_{\Gamma_{i}^{b}}=-\beta.*U$, on $\Gamma_{i}^{b}=\Gamma_{i}\cap\Gamma\neq\emptyset$, for $(i,l)\in I_{\Gamma}$. Applying both Cauchy-Schwarz and trace inequalities, it holds
   \begin{equation*}
    |<U,V>|\leq\underset{(i,l)\in I_{\Gamma}}{\sum}c_{l}\left\{\underset{k=0}{\overset{1}\sum}\int_{\Gamma_{i+k}\cap\Gamma}
|V^{t}(\beta.*U)|d\Gamma +\underset{k=1}{\overset{p}\sum}\int_{\bar{\Gamma}_{i_{k}}\cap\Gamma}|V^{t}(\beta.*U)|d\Gamma\right\}\leq
     \end{equation*}
     \begin{equation*}
  \underset{(i,l)\in I_{\Gamma}}{\sum}c_{l}\left\{\underset{k=0}{\overset{1}\sum}\left(\int_{\Gamma_{i+k}}
V^{t}Vd\Gamma\right)^{\frac{1}{2}}\left(\int_{\Gamma_{i+k}}(\beta.*U)^{t}(\beta.*U)d\Gamma\right)^{\frac{1}{2}}+\underset{k=1}{\overset{p}\sum}
 \left(\int_{\bar{\Gamma}_{i_{k}}}V^{t}Vd\Gamma\right)^{\frac{1}{2}}\times\right.
     \end{equation*}
     \begin{equation*}
     \left.\left(\int_{\bar{\Gamma}_{i_{k}}}(\beta.*U)^{t}(\beta.*U)d\Gamma\right)^{\frac{1}{2}}\right\}\leq \widehat{C}_{1}\underset{(i,l)\in I}{\sum}c_{l}\left\{\underset{k=0}{\overset{1}\sum}\left(\int_{Q_{i+k}}V^{t}Vdx\right)^{\frac{1}{2}}\left(\underset{j=1}{\overset{2}\sum}\int_{Q_{i+k}}(\nabla^{t}v_{j})(\nabla v_{j})dx\right)^{\frac{1}{2}}\times\right.
     \end{equation*}
     \begin{equation*}
     \left(\int_{Q_{i+k}}(\beta.*U)^{t} (\beta.*U)dx\right)^{\frac{1}{2}}\left(\underset{j=1}{\overset{2}\sum}\int_{Q_{i+k}}(\beta_{j}\nabla^{t}u_{j})(\beta_{j}\nabla u_{j})dx\right)^{\frac{1}{2}}+\underset{k=1}{\overset{p}\sum}\left(\int_{\bar{Q}_{i_{k}}}V^{t}Vdx\right)^{\frac{1}{2}}\times
      \end{equation*}
      \begin{equation*}
      \left.\left(\underset{j=1}{\overset{2}\sum}\int_{\bar{Q}_{i_{k}}}(\nabla^{t}v_{j})(\nabla v_{j})dx\right)^{\frac{1}{2}}\left(\int_{\bar{Q}_{i_{k}}}(\beta.*U)^{t} (\beta.*U)dx\right)^{\frac{1}{2}} \left(\underset{j=1}{\overset{2}\sum}\int_{\bar{Q}_{i_{k}}}(\beta_{j}\nabla u_{j})^{t}(\beta_{j}\nabla u_{j})dx\right)^{\frac{1}{2}}\right\}\leq
      \end{equation*}
      \begin{equation*}
      \widehat{C}_{1}(p+2)\underset{(i,l)\in I}{\sum}c_{l}\left\{\left(\int_{Q_{i}}V^{t}Vdx\right)^{\frac{1}{2}}
\left(\underset{j=1}{\overset{2}\sum}\int_{Q_{i}}(\nabla^{t}v_{j})(\nabla v_{j})dx\right)^{\frac{1}{2}}\left(\int_{Q_{i}}(\beta.*U)^{t}(\beta.*U)dx\right)^{\frac{1}{2}}\times\right.
     \end{equation*}
     \begin{equation*}
     \left.\left(\underset{j=1}{\overset{2}\sum}\int_{Q_{i}}\beta_{j}^{2}(\nabla^{t}u_{j})(\nabla u_{j})dx\right)^{\frac{1}{2}}\right\}.
      \end{equation*}
     Applying again the Cauchy-Schwarz inequality to get the second estimate in Lemma $\ref{l5}$. This completes the proof of this Lemma.
     \end{proof}

   To establish the stability and error estimates of the new predictor-corrector approach with orthogonal spline collocation FEM for solving the given FitzHugh-Nagumo system, the exact solution $w$ is assumed to lie in the Sobolev space $H^{4}(0,T;\text{\,}[H^{m}(\Omega)]^{2}\cap [L_{\sup}^{2}(\Omega)]^{2})$, where $m$ is a positive integer greater than or equal $3$, that is, there exists a positive constant $\widetilde{C}_{0}$, so that
   \begin{equation}\label{rc}
   \underset{0\leq t\leq T}{\sup}\|w(t)\|_{.}+\underset{k=1}{\overset{3}\sum}\underset{0\leq t\leq T}{\sup}\|w_{kt}(t)\|_{.}\leq \widetilde{C}_{0},
   \end{equation}
   where $w_{kt}=\frac{\partial^{k}w}{\partial t^{k}}$, for $k=1,2,3$.

     \begin{theorem}\label{t1} (Stability analysis and error estimates).
      Consider $w\in H^{4}(0,T;\text{\,}[H^{m}(\Omega)]^{2}\cap [L_{\sup}^{2}(\Omega)]^{2})$ (where $m>2$ is an integer) be the exact solution of the initial-boundary value problem $(\ref{11})$-$(\ref{13})$ and let $w_{h}(t)\in\mathcal{W}_{h}$, for $0\leq t\leq T$, be the computed one given by the developed numerical approach $(\ref{s1})$-$(\ref{s4})$, then the following inequalities are satisfied
             \begin{equation*}
   \|w_{h}^{n+1}\|_{.}^{2}+\|w_{h}^{n+\frac{1}{2}}\|_{.}^{2}+\|w_{h}^{n}\|_{.}^{2}+6\underset{s=\frac{1}{2}}{\overset{n-\frac{1}{2}}\sum}
     \frac{\tau_{s}^{2}-\tau_{s-\frac{1}{2}}^{2}}{\tau_{n-\frac{1}{2}}^{2}}\|e^{s+\frac{1}{2}}-e^{s}\|_{.}^{2}\leq
   2(\|w^{n+1}\|_{.}^{2}+\|w^{n+\frac{1}{2}}\|_{.}^{2}+\|w^{n}\|_{.}^{2})+
    \end{equation*}
      \begin{equation*}
   6\left\{C_{0}(p+2)L\underset{1\leq l\leq L}{\max}\{c_{l}\}\left(12\|w_{0}\|_{[H^{m}(\Omega)]^{2}}^{2}+15\|\widetilde{w}^{\frac{1}{2}}\|_{[H^{m}(\Omega)]^{2}}^{2}\right)h^{2m}
   +2\left(T(9^{-1}\||w_{3t}|\|_{.,\infty}^{2}+\||G-3^{-1}w_{3t}|\|_{.,\infty}^{2})+\right.\right.
    \end{equation*}
    \begin{equation*}
   \left.\left.60\||w_{2t}|\|_{.,\infty}^{2}\right)\tau_{n}^{4}\right\}\exp\left((4+10C_{F}+128\hat{C}\hat{C}_{1}^{2}(\underset{1\leq j\leq2}{\min}\gamma_{j})^{-1}
   (\underset{1\leq j\leq2}{\max}\gamma_{j})^{2}(\underset{1\leq j\leq2}{\max}\beta_{j})^{4})T\right),
     \end{equation*}
         \begin{equation*}
   \|e^{n+1}\|_{.}^{2}+\|e^{n+\frac{1}{2}}\|_{.}^{2}+\|e^{n}\|_{.}^{2}+2\underset{s=\frac{1}{2}}{\overset{n-\frac{1}{2}}\sum}
     \frac{\tau_{s}^{2}-\tau_{s-\frac{1}{2}}^{2}}{\tau_{n-\frac{1}{2}}^{2}}\|e^{s+\frac{1}{2}}-e^{s}\|_{.}^{2}\leq \left\{C_{0}(p+2)L\underset{1\leq l\leq L}{\max}\{c_{l}\}\left(12\|w_{0}\|_{[H^{m}(\Omega)]^{2}}^{2}+\right.\right.
    \end{equation*}
    \begin{equation*}
   \left.\left.15\|\widetilde{w}^{\frac{1}{2}}\|_{[H^{m}(\Omega)]^{2}}^{2}\right)h^{2m}+2\left(60\||w_{2t}|\|_{.,\infty}^{2}+
   T(9^{-1}\||w_{3t}|\|_{.,\infty}^{2}+\||G-3^{-1}w_{3t}|\|_{.,\infty}^{2})\right)\tau_{n}^{4}\right]\times
    \end{equation*}
     \begin{equation*}
    \exp\left((4+10C_{F}+128\hat{C}\hat{C}_{1}^{2}(\underset{1\leq j\leq2}{\min}\gamma_{j})^{-1}(\underset{1\leq j\leq2}{\max}\gamma_{j})^{2}(\underset{1\leq j\leq2}{\max}\beta_{j})^{4})T\right),
     \end{equation*}
    for $n=0,1,...,N-1$, where $G$ is given by equation $(\ref{22b})$, and $C_{F}$, $C_{0}$, $\hat{C}_{1}$ and $\hat{C}$ are positive constants independent of the space step $h$ and local time steps $\tau_{n}$, whereas $p$ is a fixed positive integer.
       \end{theorem}

       \begin{remark}
        The first inequality in Theorem $\ref{t1}$ suggests that the constructed algorithm $(\ref{s1})$-$(\ref{s4})$ is unconditionally stable thanks to the regularity condition $(\ref{rc})$, while the second estimate indicates that the proposed computational technique is temporal second-order accurate and spatial $mth$-order convergent.
       \end{remark}

       \begin{proof}
       Set $e^{s}=w_{h}^{s}-w^{s}$, for $s\in\{n-\frac{1}{2},n,n+\frac{1}{2},n+1\}$, be the error at time $t_{s}$. Equation $(\ref{19})$ can be rewritten as
       \begin{equation*}
    \left(\frac{\tau_{n-\frac{1}{2}}}{\tau_{n}(\tau_{n}+\tau_{n-\frac{1}{2}})}w^{n+\frac{1}{2}}+\frac{\tau_{n}-\tau_{n-\frac{1}{2}}}{\tau_{n}\tau_{n-\frac{1}{2}}}
    w^{n}-\frac{\tau_{n}}{\tau_{n-\frac{1}{2}}(\tau_{n}+\tau_{n-\frac{1}{2}})}w^{n-\frac{1}{2}},z\right)_{\cdot}+   \left(\bar{\gamma}.*\bar{\nabla}^{t}w^{n},\bar{\nabla}z\right)_{*,\cdot}=
     \end{equation*}
   \begin{equation}\label{33}
   \left(F(w^{n}),z\right)_{\cdot}-<\bar{\gamma}.*\bar{\beta}.*w^{n},z>+
   <\bar{\gamma}.*w^{n},z>_{\cdot}+\frac{\tau_{n}\tau_{n-\frac{1}{2}}}{6}\left(w_{3t}(\theta^{n}),z\right)_{\cdot}.
     \end{equation}

   Since $w_{h}^{s}=\underset{j=1}{\overset{M_{m}}\sum}w_{h,j}^{s}.*\rho_{j}$, approximation $(\ref{s1})$ can be expressed as
    \begin{equation*}
    \left(\frac{\tau_{n-\frac{1}{2}}}{\tau_{n}(\tau_{n}+\tau_{n-\frac{1}{2}})}w_{h}^{n+\frac{1}{2}}+\frac{\tau_{n}-\tau_{n-\frac{1}{2}}}{\tau_{n}\tau_{n-\frac{1}{2}}}
    w_{h}^{n}-\frac{\tau_{n}}{\tau_{n-\frac{1}{2}}(\tau_{n}+\tau_{n-\frac{1}{2}})}w_{h}^{n-\frac{1}{2}},z\right)_{\cdot}+   \left(\bar{\gamma}.*\bar{\nabla}^{t}w_{h}^{n},\bar{\nabla}z\right)_{*,\cdot}=
     \end{equation*}
   \begin{equation}\label{34}
   \left(F(w_{h}^{n}),z\right)_{\cdot}-<\bar{\gamma}.*\bar{\beta}.*w_{h}^{n},z>+
   <\bar{\gamma}.*w_{h}^{n},z>_{\cdot}.
     \end{equation}

   Plugging equations $(\ref{33})$ and $(\ref{34})$, using equality $e^{s}_{h}=w_{h}^{s}-w^{s}$, and taking $z=e^{n}$, simple computations provide
    \begin{equation*}
    \left(\tau_{n-\frac{1}{2}}^{2}(e^{n+\frac{1}{2}}-e^{n})+\tau_{n}^{2}(e^{n}-e^{n-\frac{1}{2}}),e^{n}\right)_{\cdot}+\tau_{n}\tau_{n-\frac{1}{2}}
    (\tau_{n}+\tau_{n-\frac{1}{2}})\left(\bar{\gamma}.*\bar{\nabla}^{t}e^{n},\bar{\nabla}e^{n}\right)_{*,\cdot}=\tau_{n}\tau_{n-\frac{1}{2}}(\tau_{n}+\tau_{n-\frac{1}{2}})\times
     \end{equation*}
   \begin{equation}\label{35}
   [\left(F(w_{h}^{n})-F(w^{n}),e^{n}\right)_{\cdot}-<\bar{\gamma}.*\bar{\beta}.*e^{n},e^{n}>+
   <\bar{\gamma}.*e^{n},e^{n}>_{\cdot}]-\frac{\tau_{n}^{2}\tau_{n-\frac{1}{2}}^{2}(\tau_{n}+\tau_{n-\frac{1}{2}})}{6}\left(w_{3t}(\theta^{n}),e^{n}\right)_{\cdot}.
     \end{equation}
   But
       \begin{equation*}
     \left(\tau_{n-\frac{1}{2}}^{2}(e^{n+\frac{1}{2}}-e^{n})+\tau_{n}^{2}(e^{n}-e^{n-\frac{1}{2}}),e^{n}\right)_{\cdot}=\frac{1}{2}\left[\tau_{n-\frac{1}{2}}^{2}
     (\|e^{n+\frac{1}{2}}\|_{.}^{2}-\|e^{n+\frac{1}{2}}-e^{n}\|_{.}^{2}-\|e^{n}\|_{.}^{2})+\tau_{n}^{2}\|e^{n-\frac{1}{2}}-e^{n}\|_{.}^{2}+\right.
       \end{equation*}
       \begin{equation*}
       \left.\tau_{n}^{2}(\|e^{n}\|_{.}^{2}-\|e^{n-\frac{1}{2}}\|_{.}^{2})\right]=\frac{1}{2}\left[\tau_{n-\frac{1}{2}}^{2}(\|e^{n+\frac{1}{2}}\|_{.}^{2}-\|e^{n}\|_{.}^{2})+
     \tau_{n}^{2}(\|e^{n}\|_{.}^{2}-\|e^{n-\frac{1}{2}}\|_{.}^{2})\right]+
       \end{equation*}
       \begin{equation*}
       \frac{1}{2}\left[\tau_{n}^{2}\|e^{n-\frac{1}{2}}-e^{n}\|_{.}^{2}-\tau_{n-\frac{1}{2}}^{2}\|e^{n+\frac{1}{2}}-e^{n}\|_{.}^{2}\right].
       \end{equation*}

       Substituting this into equation $(\ref{35})$, summing up the obtained equation for $s=\frac{1}{2},1,...,n-\frac{1}{2},n$ (step size of summation index is $\frac{1}{2}$), to get
        \begin{equation*}
    \frac{1}{2}\underset{s=\frac{1}{2}}{\overset{n}\sum}\left[\tau_{s-\frac{1}{2}}^{2}(\|e^{s+\frac{1}{2}}\|_{.}^{2}-\|e^{s}\|_{.}^{2})+\tau_{s}^{2}(\|e^{s}\|_{.}^{2}-
    \|e^{s-\frac{1}{2}}\|_{.}^{2})+\tau_{s}^{2}\|e^{s-\frac{1}{2}}-e^{s}\|_{.}^{2}-\tau_{s-\frac{1}{2}}^{2}\|e^{s+\frac{1}{2}}-e^{s}\|_{.}^{2}\right]+
     \end{equation*}
     \begin{equation*}
    \underset{s=\frac{1}{2}}{\overset{n}\sum}\tau_{s}\tau_{s-\frac{1}{2}}(\tau_{s}+\tau_{s-\frac{1}{2}})\left(\bar{\gamma}.*\bar{\nabla}^{t}e^{s},
    \bar{\nabla}e^{s}\right)_{*,\cdot}=\underset{s=\frac{1}{2}}{\overset{n}\sum}\tau_{s}\tau_{s-\frac{1}{2}}(\tau_{s}+\tau_{s-\frac{1}{2}})
   [\left(F(w_{h}^{s})-F(w^{s}),e^{s}\right)_{\cdot}-<\bar{\gamma}.*\bar{\beta}.*e^{s},e^{s}>+
     \end{equation*}
   \begin{equation}\label{36}
   <\bar{\gamma}.*e^{s},e^{s}>_{\cdot}]-\frac{1}{6} \underset{s=\frac{1}{2}}{\overset{n}\sum}\tau_{s}^{2}\tau_{s-\frac{1}{2}}^{2}(\tau_{s}+
   \tau_{s-\frac{1}{2}})\left(w_{3t}(\theta^{s}),e^{s}\right)_{\cdot}.
     \end{equation}

    Utilizing the summation by parts, it is not hard to show that
    \begin{equation}\label{37}
    \underset{s=\frac{1}{2}}{\overset{n}\sum}\tau_{s-\frac{1}{2}}^{2}(\|e^{s+\frac{1}{2}}\|_{.}^{2}-\|e^{s}\|_{.}^{2})=
    \tau_{n-\frac{1}{2}}^{2}\|e^{n+\frac{1}{2}}\|_{.}^{2}-\tau_{0}^{2}\|e^{\frac{1}{2}}\|_{.}^{2}-
     \underset{s=0}{\overset{n-1}\sum}(\tau_{s+\frac{1}{2}}^{2}-\tau_{s}^{2})\|e^{s+1}\|_{.}^{2},
     \end{equation}
     \begin{equation}\label{38}
    \underset{s=\frac{1}{2}}{\overset{n}\sum}\tau_{s}^{2}(\|e^{s}\|_{.}^{2}-\|e^{s-\frac{1}{2}}\|_{.}^{2})=
    \tau_{n}^{2}\|e^{n}\|_{.}^{2}-\tau_{\frac{1}{2}}^{2}\|e^{0}\|_{.}^{2}-
     \underset{s=\frac{1}{2}}{\overset{n}\sum}(\tau_{s+\frac{1}{2}}^{2}-\tau_{s}^{2})\|e^{s}\|_{.}^{2},
     \end{equation}
       \begin{equation}\label{39}
    \underset{s=\frac{1}{2}}{\overset{n}\sum}(\tau_{s}^{2}\|e^{s}-e^{s-\frac{1}{2}}\|_{.}^{2}-\tau_{s-\frac{1}{2}}^{2}\|e^{s+\frac{1}{2}}-e^{s}\|_{.}^{2})=
    -\tau_{n-\frac{1}{2}}^{2}\|e^{n+\frac{1}{2}}-e^{n}\|_{.}^{2}+\tau_{\frac{1}{2}}^{2}\|e^{\frac{1}{2}}-e^{0}\|_{.}^{2}+
     \underset{s=\frac{1}{2}}{\overset{n-\frac{1}{2}}\sum}(\tau_{s+\frac{1}{2}}^{2}-\tau_{s-\frac{1}{2}}^{2})\|e^{s+\frac{1}{2}}-e^{s}\|_{.}^{2}.
       \end{equation}

   Since $\tau_{s+\frac{1}{2}}=\tau_{s}$, combining equations $(\ref{36})$-$(\ref{39})$, this provides
       \begin{equation*}
       \frac{1}{2}[\tau_{n-\frac{1}{2}}^{2}\|e^{n+\frac{1}{2}}\|_{.}^{2}+\tau_{n}^{2}\|e^{n}\|_{.}^{2}-\tau_{n-\frac{1}{2}}^{2}\|e^{n+\frac{1}{2}}-e^{n}\|_{.}^{2}
       +\underset{s=\frac{1}{2}}{\overset{n-\frac{1}{2}}\sum}(\tau_{s}^{2}-\tau_{s-\frac{1}{2}}^{2})\|e^{s+\frac{1}{2}}-e^{s}\|_{.}^{2}]+
       \end{equation*}
     \begin{equation*}
    \underset{s=\frac{1}{2}}{\overset{n}\sum}\tau_{s}\tau_{s-\frac{1}{2}}(\tau_{s}+\tau_{s-\frac{1}{2}})\left(\bar{\gamma}.*\bar{\nabla}^{t}e^{s},
    \bar{\nabla}e^{s}\right)_{*,\cdot} \leq \frac{\tau_{0}^{2}}{2}(\|e^{0}\|_{.}^{2}+\|e^{\frac{1}{2}}\|_{.}^{2})
    +\underset{s=\frac{1}{2}}{\overset{n}\sum}\tau_{s}\tau_{s-\frac{1}{2}}(\tau_{s}+\tau_{s-\frac{1}{2}})[\left(F(w_{h}^{s})-F(w^{s}),e^{s}\right)_{\cdot}-
     \end{equation*}
   \begin{equation}\label{40}
   <\bar{\gamma}.*\bar{\beta}.*e^{s},e^{s}>+<\bar{\gamma}.*e^{s},e^{s}>_{\cdot}]-\frac{1}{6} \underset{s=\frac{1}{2}}{\overset{n}\sum}\tau_{s}^{2}
    \tau_{s-\frac{1}{2}}^{2}(\tau_{s}+\tau_{s-\frac{1}{2}})\left(w_{3t}(\theta^{s}),e^{s}\right)_{\cdot}.
     \end{equation}

   In a similar manner, using equations $(\ref{22})$ and $(\ref{s2})$, and taking $z=e^{n+1}$, one easily shows that
        \begin{equation*}
     \frac{1}{2}[\frac{2}{3}\|e^{n+1}-e^{n+\frac{1}{2}}\|_{.}^{2}+\|e^{n+1}\|_{.}^{2}-\frac{1}{3}\|e^{n+\frac{1}{2}}\|_{.}^{2}]
       +\frac{2}{3}\underset{s=0}{\overset{n}\sum}\tau_{s}\left(\bar{\gamma}.*\bar{\nabla}^{t}e^{s+1},\bar{\nabla}e^{s+1}\right)_{*,\cdot}\leq
       \end{equation*}
     \begin{equation*}
     \frac{1}{2}(\|e^{\frac{1}{2}}-e^{0}\|_{.}^{2}+\|e^{\frac{1}{2}}\|_{.}^{2})+\frac{2}{3}\underset{s=0}{\overset{n}\sum}\tau_{s}\left[\left(2(F(w_{h}^{s+\frac{1}{2}})
   -F(w^{s+\frac{1}{2}}))-(F(w_{h}^{s})-F(w^{s})),e^{s+1}\right)_{\cdot}\right.-
     \end{equation*}
   \begin{equation}\label{41}
   \left.<\bar{\gamma}.*\bar{\beta}.*e^{s+1},e^{s+1}>+<\bar{\gamma}.*e^{s+1},e^{s+1}>_{\cdot}\right]+\frac{2}{3}\underset{s=\frac{1}{2}}{\overset{n}\sum}\tau_{s}^{3}
   \left(G^{s},e^{s+1}\right)_{\cdot}-\frac{2}{9}\underset{s=0}{\overset{n}\sum}\tau_{s}^{3}\left(w_{3t}(\theta^{s+1}),e^{s+1}\right)_{\cdot}.
     \end{equation}

    Plugging estimates $(\ref{40})$ and $(\ref{41})$, and using the fact that the sequence of local time steps $\{\tau_{s},\text{\,}s=0,\frac{1}{2},...,N\}$ is nondecreasing, straightforward computations result in
       \begin{equation*}
    \frac{3}{2}\|e^{n+1}\|_{.}^{2}+\frac{1}{2}\|e^{n+\frac{1}{2}}\|_{.}^{2}+\|e^{n}\|_{.}^{2}+\|e^{n+1}-e^{n+\frac{1}{2}}\|_{.}^{2}-\|e^{n+\frac{1}{2}}-e^{n}\|_{.}^{2}
       +\underset{s=\frac{1}{2}}{\overset{n-\frac{1}{2}}\sum}\frac{\tau_{s}^{2}-\tau_{s-\frac{1}{2}}^{2}}{\tau_{n-\frac{1}{2}}^{2}}\|e^{s+\frac{1}{2}}-e^{s}\|_{.}^{2}+
       \end{equation*}
   \begin{equation*}
    \tau_{0}\left[\underset{s=\frac{1}{2}}{\overset{n}\sum}\left(\bar{\gamma}.*\bar{\nabla}^{t}e^{s},\bar{\nabla}e^{s}\right)_{*,\cdot}+
     \underset{s=0}{\overset{n}\sum}\left(\bar{\gamma}.*\bar{\nabla}^{t}e^{s+1},\bar{\nabla}e^{s+1}\right)_{*,\cdot}\right]\leq \frac{1}{2}(2\|e^{0}\|_{.}^{2}+5\|e^{\frac{1}{2}}\|_{.}^{2}+3\|e^{\frac{1}{2}}-e^{0}\|_{.}^{2})+2\underset{s=0}{\overset{n}\sum}\tau_{s}\times
       \end{equation*}
     \begin{equation*}
   \left[\left(2(F(w_{h}^{s+\frac{1}{2}})-F(w^{s+\frac{1}{2}}))-(F(w_{h}^{s})-F(w^{s})),e^{s+1}\right)_{\cdot}-
   <\bar{\gamma}.*\bar{\beta}.*e^{s+1},e^{s+1}>+<\bar{\gamma}.*e^{s+1},e^{s+1}>_{\cdot}\right]
     \end{equation*}
     \begin{equation*}
   +2\underset{s=\frac{1}{2}}{\overset{n}\sum}\frac{\tau_{s}\tau_{s-\frac{1}{2}}(\tau_{s}+\tau_{s-\frac{1}{2}})}
   {\tau_{n-\frac{1}{2}}^{2}}\left[\left(F(w_{h}^{s})-F(w^{s}),e^{s}\right)_{\cdot}-<\bar{\gamma}.*\bar{\beta}.*e^{s},e^{s}>+<\bar{\gamma}.*e^{s},e^{s}>_{\cdot}\right]
     \end{equation*}
   \begin{equation}\label{42}
    +2\underset{s=0}{\overset{n}\sum}\tau_{s}^{3}\left(G^{s}-\frac{1}{3}w_{3t}(\theta^{s+1}),e^{s+1}\right)_{\cdot}-
   \frac{2\tau_{0}^{3}}{3}\left(w_{3t}(\theta^{1}),e^{1}\right)_{\cdot}+\frac{1}{3}\underset{s=\frac{1}{2}}{\overset{n}\sum}\frac{\tau_{s}^{2}\tau_{s-\frac{1}{2}}^{2}(\tau_{s}+\tau_{s-\frac{1}{2}})}
   {\tau_{n-\frac{1}{2}}^{2}}\left(w_{3t}(\theta^{s}),e^{s}\right)_{\cdot}.
     \end{equation}
   Indeed: since $\bar{\gamma}=(\gamma_{1},\gamma_{2})^{t}$, where $\gamma_{j}\geq0$ ($j=1,2$), it is not difficult to see that $\left(\bar{\gamma}.*\bar{\nabla}^{t}e^{r},\bar{\nabla}e^{r}\right)_{*,\cdot}\geq0$, for $r=s,s+1$. If the nonnegative "generalized sequence" $\{\|e^{s+1}-e^{s+\frac{1}{2}}\|,\text{\,}s=0,\frac{1}{2},...,N-1\}$, is increasing then $\|e^{n+1}-e^{n+\frac{1}{2}}\|_{.}^{2}-\|e^{n+\frac{1}{2}}-e^{n}\|_{.}^{2}\geq0$, if this sequence is decreasing it holds $\|e^{s+1}-e^{s+\frac{1}{2}}\|_{.}\leq\|e^{\frac{1}{2}}-e^{0}\|_{.}$, for $s=0,\frac{1}{2},...,N-1$. Otherwise, there is a subsequence of the "generalized sequence" that is either increasing or decreasing. Let denote by $\{\|e^{s+1}-e^{s+\frac{1}{2}}\|,\text{\,}s=0,\frac{1}{2},...,N-1\}$ be such a subsequence. However, $\|e^{n+1}-e^{n+\frac{1}{2}}\|_{.}^{2}-\|e^{n+\frac{1}{2}}-e^{n}\|_{.}^{2}\geq0$ or $\|e^{n+1}-e^{n+\frac{1}{2}}\|_{.}^{2}-\|e^{n+\frac{1}{2}}-e^{n}\|_{.}^{2}\geq \|e^{n+1}-e^{n+\frac{1}{2}}\|_{.}^{2}-\|e^{\frac{1}{2}}-e^{0}\|_{.}^{2}$, for any $n=0,1,...,N-1$. Indeed, both estimates are due to the fact that if the "generalized sequence" is neither increasing nor decreasing, one should consider a monotone subsequence of this sequence. Using these facts, estimate $(\ref{42})$ implies
    \begin{equation*}
    \frac{3}{2}\|e^{n+1}\|_{.}^{2}+\frac{1}{2}\|e^{n+\frac{1}{2}}\|_{.}^{2}+\|e^{n}\|_{.}^{2}+\underset{s=\frac{1}{2}}{\overset{n-\frac{1}{2}}\sum}
     \frac{\tau_{s}^{2}-\tau_{s-\frac{1}{2}}^{2}}{\tau_{n-\frac{1}{2}}^{2}}\|e^{s+\frac{1}{2}}-e^{s}\|_{.}^{2}+
    \tau_{0}\left[\underset{s=\frac{1}{2}}{\overset{n}\sum}\left(\bar{\gamma}.*\bar{\nabla}^{t}e^{s},\bar{\nabla}e^{s}\right)_{*,\cdot}+\right.
       \end{equation*}
   \begin{equation*}
    \left.\underset{s=0}{\overset{n}\sum}\left(\bar{\gamma}.*\bar{\nabla}^{t}e^{s+1},\bar{\nabla}e^{s+1}\right)_{*,\cdot}\right]\leq \|e^{0}\|_{.}^{2}+\frac{5}{2}\|e^{\frac{1}{2}}\|_{.}^{2}+\frac{5}{2}\|e^{\frac{1}{2}}-e^{0}\|_{.}^{2}+4\underset{s=0}{\overset{n}\sum}\tau_{s}
     \left(F(w_{h}^{s+\frac{1}{2}})-F(w^{s+\frac{1}{2}}),e^{s+1}\right)_{\cdot}
       \end{equation*}
     \begin{equation*}
   -2\underset{s=0}{\overset{n}\sum}\tau_{s}\left[\left(F(w_{h}^{s})-F(w^{s}),e^{s+1}\right)_{\cdot}+<\bar{\gamma}.*\bar{\beta}.*e^{s+1},e^{s+1}>-
    <\bar{\gamma}.*e^{s+1},e^{s+1}>_{\cdot}\right]
     \end{equation*}
     \begin{equation*}
   +2\underset{s=\frac{1}{2}}{\overset{n}\sum}\frac{\tau_{s}\tau_{s-\frac{1}{2}}(\tau_{s}+\tau_{s-\frac{1}{2}})}
   {\tau_{n-\frac{1}{2}}^{2}}\left[\left(F(w_{h}^{s})-F(w^{s}),e^{s}\right)_{\cdot}-<\bar{\gamma}.*\bar{\beta}.*e^{s},e^{s}>+<\bar{\gamma}.*e^{s},e^{s}>_{\cdot}\right]
     \end{equation*}
   \begin{equation}\label{43}
    +2\underset{s=0}{\overset{n}\sum}\tau_{s}^{3}\left(G^{s}-\frac{1}{3}w_{3t}(\theta^{s+1}),e^{s+1}\right)_{\cdot}-
   \frac{2\tau_{0}^{3}}{3}\left(w_{3t}(\theta^{1}),e^{1}\right)_{\cdot}+\frac{1}{3}\underset{s=\frac{1}{2}}{\overset{n}\sum}\frac{\tau_{s}^{2}
   \tau_{s-\frac{1}{2}}^{2}(\tau_{s}+\tau_{s-\frac{1}{2}})}{\tau_{n-\frac{1}{2}}^{2}}\left(w_{3t}(\theta^{s}),e^{s}\right)_{\cdot}.
     \end{equation}

    Combining Lemma $\ref{l4}$ and Cauchy-Schwarz inequality, direct calculations yield
  \begin{equation}\label{44}
   \left(2(F(w_{h}^{s+\frac{1}{2}})-F(w^{s+\frac{1}{2}}))-(F(w_{h}^{s})-F(w^{s})),e^{s+1}\right)_{\cdot}\leq C_{F}(\frac{3}{2}\|e^{s+1}\|_{.}^{2}+
   \|e^{s+\frac{1}{2}}\|_{.}^{2}+\frac{1}{2}\|e^{s}\|_{.}^{2}),
     \end{equation}
   \begin{equation}\label{45}
   \left(F(w_{h}^{s})-F(w^{s}),e^{s}\right)_{\cdot}\leq C_{F}\|e^{s}\|_{.}^{2},
     \end{equation}
    \begin{equation}\label{46}
   2\tau_{s}^{3}\left(G^{s}-\frac{1}{3}w_{3t}(\theta^{s+1}),e^{s+1}\right)_{\cdot}\leq 2\tau_{s}^{3}\|e^{s+1}\|_{.}\|G^{s}-\frac{1}{3}w_{3t}(\theta^{s+1})\|_{.}\leq
   \tau_{s}\|e^{s+1}\|_{.}^{2}+\tau_{s}^{5}\|G^{s}-\frac{1}{3}w_{3t}(\theta^{s+1})\|_{.}^{2},
     \end{equation}
   \begin{equation}\label{47}
   -\frac{2}{3}\tau_{0}^{3}\left(w_{3t}(\theta^{1}),e^{1}\right)_{\cdot}\leq \frac{2}{3}\tau_{0}^{3}\|e^{1}\|_{.}\|w_{3t}(\theta^{1})\|_{.}\leq
   \tau_{0}\|e^{1}\|_{.}^{2}+\frac{\tau_{0}^{5}}{9}\|w_{3t}(\theta^{1})\|_{.}^{2},
     \end{equation}
   \begin{equation}\label{48}
   \frac{1}{3}\tau_{s}^{2}\tau_{s-\frac{1}{2}}^{2}\tau_{n-\frac{1}{2}}^{-2}(\tau_{s}+\tau_{s-\frac{1}{2}})\left(w_{3t}(\theta^{s}),e^{s}\right)_{\cdot}\leq \frac{2}{3}\tau_{s}^{3}\|e^{s}\|_{.}\|w_{3t}(\theta^{s})\|_{.}\leq \tau_{s}\|e^{s}\|_{.}^{2}+\frac{\tau_{s}^{5}}{9}\|w_{3t}(\theta^{s})\|_{.}^{2}.
     \end{equation}

   Utilizing Lemma $\ref{l5}$, direct computations results in

    \begin{equation*}
   -2\tau_{s}<\bar{\gamma}.*\bar{\beta}.*e^{s+1},e^{s+1}>\leq 2\tau_{s}|<\bar{\beta}.*e^{s+1},\bar{\gamma}.*e^{s+1}>|\leq 8\hat{C}_{1}\tau_{s}
    \|\bar{\gamma}.*e^{s+1}\|_{.}^{\frac{1}{2}}\|\bar{\nabla}^{t}(\bar{\gamma}.*e^{s+1})\|_{*,.}^{\frac{1}{2}}
     \end{equation*}
     \begin{equation*}
   \|\bar{\beta}.*\bar{\beta}.*e^{s+1}\|_{.}^{\frac{1}{2}}\|\bar{\nabla}^{t}(\bar{\beta}.*\bar{\beta}.*e^{s+1})\|_{*,.}^{\frac{1}{2}}\leq 8(\underset{1\leq j\leq2}{\max}\gamma_{j}\underset{1\leq j\leq2}{\max}\beta_{j}^{2})\hat{C}_{1}\tau_{s}\|e^{s+1}\|_{.}\|\bar{\nabla}^{t}e^{s+1}\|_{*,.}\leq
     \end{equation*}
   \begin{equation}\label{49}
   \tau_{0}\underset{1\leq j\leq2} {\min}\gamma_{j}\|\bar{\nabla}^{t}e^{s+1}\|_{*,.}^{2}+ 16\hat{C}_{1}^{2}\tau_{s}^{2}\tau_{0}^{-1}(\underset{1\leq j\leq2} {\min}\gamma_{j})^{-1}(\underset{1\leq j\leq2}{\max}\gamma_{j})^{2}(\underset{1\leq j\leq2}{\max}\beta_{j})^{4}\|e^{s+1}\|_{.}^{2},
     \end{equation}

  \begin{equation*}
   -2\frac{\tau_{s}\tau_{s-\frac{1}{2}}(\tau_{s}+\tau_{s-\frac{1}{2}})}{\tau_{n-\frac{1}{2}}^{2}}<\bar{\gamma}.*\bar{\beta}.*e^{s},e^{s}>\leq 4\frac{\tau_{s}^{2}\tau_{s-\frac{1}{2}}}{\tau_{n-\frac{1}{2}}^{2}}|<\bar{\beta}.*e^{s},\bar{\gamma}.*e^{s}>|\leq 16\hat{C}_{1}\tau_{s}
    \|\bar{\gamma}.*e^{s}\|_{.}^{\frac{1}{2}}\|\bar{\nabla}^{t}(\bar{\gamma}.*e^{s})\|_{*,.}^{\frac{1}{2}}
     \end{equation*}
   \begin{equation}\label{50}
\|\bar{\beta}.*\bar{\beta}.*e^{s}\|_{.}^{\frac{1}{2}}\|\bar{\nabla}^{t}(\bar{\beta}.*\bar{\beta}.*e^{s})\|_{*,.}^{\frac{1}{2}}\leq
   \tau_{0}\underset{1\leq j\leq2}{\min}\gamma_{j}\|\bar{\nabla}^{t}e^{s}\|_{*,.}^{2}+ 64\hat{C}_{1}^{2}\tau_{s}^{2}\tau_{0}^{-1}(\underset{1\leq j\leq2} {\min}\gamma_{j})^{-1}(\underset{1\leq j\leq2}{\max}\gamma_{j})^{2}(\underset{1\leq j\leq2}{\max}\beta_{j})^{4}\|e^{s}\|_{.}^{2},
     \end{equation}
   \begin{equation}\label{51}
    \|e^{\frac{1}{2}}-e^{0}\|_{.}^{2}\leq2(\|e^{\frac{1}{2}}\|_{.}^{2}+\|e^{0}\|_{.}^{2}),\text{\,\,}<\bar{\gamma}.e^{r},e^{r}>_{.}=0,\text{\,\,\,}
    \left(\bar{\gamma}.\bar{\nabla}^{t}e^{r},\bar{\nabla}^{t}e^{r}\right)_{*,.}\geq(\underset{1\leq j\leq2} {\min}\gamma_{j})\|\bar{\nabla}^{t}e^{r}\|_{*,.}^{2},\text{\,\,for\,\,}r=s,s+1.
     \end{equation}

   Substituting estimates $(\ref{44})$-$(\ref{51})$ into equation $(\ref{43})$ and rearranging terms, this yields
     \begin{equation*}
    \frac{3}{2}\|e^{n+1}\|_{.}^{2}+\frac{1}{2}\|e^{n+\frac{1}{2}}\|_{.}^{2}+\|e^{n}\|_{.}^{2}+\underset{s=\frac{1}{2}}{\overset{n-\frac{1}{2}}\sum}
     \frac{\tau_{s}^{2}-\tau_{s-\frac{1}{2}}^{2}}{\tau_{n-\frac{1}{2}}^{2}}\|e^{s+\frac{1}{2}}-e^{s}\|_{.}^{2}\leq 6\|e^{0}\|_{.}^{2}+\frac{15}{2}\|e^{\frac{1}{2}}\|_{.}^{2}+\underset{s=0}{\overset{n}\sum}\tau_{s}\left[\left(2+3C_{F}+\right.\right.
    \end{equation*}
     \begin{equation*}
   \left.\left. 16\hat{C}_{1}^{2}\tau_{s}^{2}\tau_{0}^{-1}(\underset{1\leq j\leq2}{\min}\gamma_{j})^{-1}(\underset{1\leq j\leq2}{\max}\gamma_{j})^{2}(\underset{1\leq j\leq2}{\max}\beta_{j})^{4}\right)\|e^{s+1}\|_{.}^{2}+2C_{F}\|e^{s+\frac{1}{2}}\|_{.}^{2}+\left(1+5C_{F}+\right.\right.
     \end{equation*}
     \begin{equation*}
   \left.\left.64\hat{C}_{1}^{2}\tau_{s}^{2}\tau_{0}^{-1}(\underset{1\leq j\leq2}{\min}\gamma_{j})^{-1}(\underset{1\leq j\leq2}{\max}\gamma_{j})^{2}(\underset{1\leq j\leq2}{\max}\beta_{j})^{4}\right)\|e^{s}\|_{.}^{2}\right]+\underset{s=0}{\overset{n}\sum}\tau_{s}^{5}\left(\|G^{s}-\frac{1}{3}w_{3t}(\theta^{s+1})\|_{.}^{2}+
    \frac{1}{9}\|w_{3t}(\theta^{s})\|_{.}^{2}\right).
     \end{equation*}

   Since $\tau_{s}\leq\tau_{n}$ and $\tau_{n}\tau_{0}^{-1}\leq\hat{C}$, for $s=0,\frac{1}{2},...,n$, this inequality implies
     \begin{equation*}
   \|e^{n+1}\|_{.}^{2}+\|e^{n+\frac{1}{2}}\|_{.}^{2}+\|e^{n}\|_{.}^{2}+2\underset{s=\frac{1}{2}}{\overset{n-\frac{1}{2}}\sum}
     \frac{\tau_{s}^{2}-\tau_{s-\frac{1}{2}}^{2}}{\tau_{n-\frac{1}{2}}^{2}}\|e^{s+\frac{1}{2}}-e^{s}\|_{.}^{2}\leq 12\|e^{0}\|_{.}^{2}+15\|e^{\frac{1}{2}}\|_{.}^{2}+2\left[2+5C_{F}+\right.
    \end{equation*}
     \begin{equation*}
    64\hat{C}\hat{C}_{1}^{2}(\underset{1\leq j\leq2}{\min}\gamma_{j})^{-1}(\underset{1\leq j\leq2}{\max}\gamma_{j})^{2}(\underset{1\leq j\leq2}{\max}\beta_{j})^{4}]\underset{s=0}{\overset{n}\sum}\tau_{s}(\|e^{s+1}\|_{.}^{2}+\|e^{s+\frac{1}{2}}\|_{.}^{2}+\|e^{s}\|_{.}^{2})+
     \end{equation*}
     \begin{equation*}
   2\underset{s=0}{\overset{n}\sum}\tau_{s}^{5}\left(\|G^{s}-\frac{1}{3}w_{3t}(\theta^{s+1})\|_{.}^{2}+\frac{1}{9}\|w_{3t}(\theta^{s})\|_{.}^{2}\right).
     \end{equation*}

   The application of the Gronwall inequality gives
      \begin{equation*}
   \|e^{n+1}\|_{.}^{2}+\|e^{n+\frac{1}{2}}\|_{.}^{2}+\|e^{n}\|_{.}^{2}+2\underset{s=\frac{1}{2}}{\overset{n-\frac{1}{2}}\sum}
     \frac{\tau_{s}^{2}-\tau_{s-\frac{1}{2}}^{2}}{\tau_{n-\frac{1}{2}}^{2}}\|e^{s+\frac{1}{2}}-e^{s}\|_{.}^{2}\leq 2\left[6\|e^{0}\|_{.}^{2}+\frac{15}{2}\|e^{\frac{1}{2}}\|_{.}^{2}+
     \underset{s=0}{\overset{n}\sum}\tau_{s}^{5}\left(\frac{1}{9}\|w_{3t}(\theta^{s})\|_{.}^{2} \right.\right.
    \end{equation*}
     \begin{equation}\label{52}
    \left.\left. +\|G^{s}-\frac{1}{3}w_{3t}(\theta^{s+1})\|_{.}^{2}\right)\right]\exp\left((4+10C_{F}+128\hat{C}\hat{C}_{1}^{2}(\underset{1\leq j\leq2}{\min}\gamma_{j})^{-1}(\underset{1\leq j\leq2}{\max}\gamma_{j})^{2}(\underset{1\leq j\leq2}{\max}\beta_{j})^{4})\underset{s=0}{\overset{n}\sum}\tau_{s}\right).
     \end{equation}

   Using the initial conditions $(\ref{s3})$, Lemma $\ref{l3}$, and equations $(\ref{25})$-$(\ref{27})$, it is not difficult to observe that
     \begin{equation}\label{53}
   \|e^{0}\|_{.}^{2}=\|w_{h}^{0}-w_{0}\|_{.}^{2}=\|P_{h}w_{0}-w_{0}\|_{.}^{2}\leq C_{0}(p+2)L\underset{1\leq l\leq L}{\max}\{c_{l}\}h^{2m}\|w_{0}\|_{[H^{m}(\Omega)]^{2}}^{2},
    \end{equation}
     \begin{equation}\label{53a}
   \|e^{\frac{1}{2}}\|_{.}^{2}\leq 2(\|P_{h}\widetilde{w}^{\frac{1}{2}}-\widetilde{w}^{\frac{1}{2}}\|_{.}^{2}+\|\widetilde{w}^{\frac{1}{2}}-w^{\frac{1}{2}}\|_{.}^{2})
    \leq 4\tau_{0}^{4}\|w_{2t}(\theta^{0})+w_{2t}(\theta^{\frac{1}{2}})\|_{.}^{2}+C_{0}(p+2)L\underset{1\leq l\leq L}{\max}\{c_{l}\}h^{2m}\|\widetilde{w}^{\frac{1}{2}}\|_{[H^{m}(\Omega)]^{2}}^{2},
     \end{equation}

   Because $\underset{s=0}{\overset{n}\sum}\tau_{s}\leq\underset{s=0}{\overset{N}\sum}\tau_{s}=T$ and $\tau_{s}\leq\tau_{n}$, for $s=0,\frac{1}{2},...,n$, using this fact together with estimates $(\ref{52})$-$(\ref{53a})$, and rearranging terms, result in
      \begin{equation*}
   \|e^{n+1}\|_{.}^{2}+\|e^{n+\frac{1}{2}}\|_{.}^{2}+\|e^{n}\|_{.}^{2}+2\underset{s=\frac{1}{2}}{\overset{n-\frac{1}{2}}\sum}
     \frac{\tau_{s}^{2}-\tau_{s-\frac{1}{2}}^{2}}{\tau_{n-\frac{1}{2}}^{2}}\|e^{s+\frac{1}{2}}-e^{s}\|_{.}^{2}\leq \left\{C_{0}(p+2)L\underset{1\leq l\leq L}{\max}\{c_{l}\}\left(12\|w_{0}\|_{[H^{m}(\Omega)]^{2}}^{2}+\right.\right.
    \end{equation*}
    \begin{equation*}
   \left.\left.15\|\widetilde{w}^{\frac{1}{2}}\|_{[H^{m}(\Omega)]^{2}}^{2}\right)h^{2m}+2\left(60\||w_{2t}|\|_{.,\infty}^{2}+
   T(\frac{1}{9}\||w_{3t}|\|_{.,\infty}^{2}+\||G-\frac{1}{3}w_{3t}|\|_{.,\infty}^{2})\right)\tau_{n}^{4}\right]\times
    \end{equation*}
     \begin{equation}\label{54}
    \exp\left((4+10C_{F}+128\hat{C}\hat{C}_{1}^{2}(\underset{1\leq j\leq2}{\min}\gamma_{j})^{-1}(\underset{1\leq j\leq2}{\max}\gamma_{j})^{2}(\underset{1\leq j\leq2}{\max}\beta_{j})^{4})T\right),
     \end{equation}
   where the norm $\||\cdot|\|_{.,\infty}$, is defined as: $\||U|\|_{.,\infty}=\underset{0\leq t\leq T}{\sup}\|U(t)\|_{.}$, for every $U\in L^{\infty}(0,T;\text{\,}[L^{2}(\Omega)]^{2})$. We remind that the exact solution $w\in H^{4}(0,T;\text{\,}[H^{m}(\Omega)]^{2}\cap[L_{\sup}^{2}(\Omega)]^{2})$, the vector-valued function $G$ is given by equation  $(\ref{22b})$, and $p$ is the nonnegative integer provided by the assumptions satisfied by the triangulation $\mathcal{F}_{h}$. Since $|\|w_{h}^{s}\|_{.}-\|w^{s}\|_{.}|\leq \|w_{h}^{s}-w^{s}\|_{.}=\|e^{s}\|_{.}$, for $s=n,n+\frac{1}{2},n+1$, direct calculations show that estimate $(\ref{54})$ implies
      \begin{equation*}
   \|w_{h}^{n+1}\|_{.}^{2}+\|w_{h}^{n+\frac{1}{2}}\|_{.}^{2}+\|w_{h}^{n}\|_{.}^{2}+6\underset{s=\frac{1}{2}}{\overset{n-\frac{1}{2}}\sum}
     \frac{\tau_{s}^{2}-\tau_{s-\frac{1}{2}}^{2}}{\tau_{n-\frac{1}{2}}^{2}}\|e^{s+\frac{1}{2}}-e^{s}\|_{.}^{2}\leq
   2(\|w^{n+1}\|_{.}^{2}+\|w^{n+\frac{1}{2}}\|_{.}^{2}+\|w^{n}\|_{.}^{2})+
    \end{equation*}
      \begin{equation*}
   6\left\{C_{0}(p+2)L\underset{1\leq l\leq L}{\max}\{c_{l}\}\left(12\|w_{0}\|_{[H^{m}(\Omega)]^{2}}^{2}+
   15\|\widetilde{w}^{\frac{1}{2}}\|_{[H^{m}(\Omega)]^{2}}^{2}\right)h^{2m}
   +2\left(T(\frac{1}{9}\||w_{3t}|\|_{.,\infty}^{2}+\||G-\frac{1}{3}w_{3t}|\|_{.,\infty}^{2})+\right.\right.
    \end{equation*}
    \begin{equation*}
   \left.\left.60\||w_{2t}|\|_{.,\infty}^{2}\right)\tau_{n}^{4}\right\}\exp\left((4+10C_{F}+128\hat{C}\hat{C}_{1}^{2}(\underset{1\leq j\leq2}{\min}\gamma_{j})^{-1}
   (\underset{1\leq j\leq2}{\max}\gamma_{j})^{2}(\underset{1\leq j\leq2}{\max}\beta_{j})^{4})T\right).
     \end{equation*}

   This completes the proof of Theorem $\ref{t1}$.
       \end{proof}

   \section{Numerical experiments}\label{sec4}
    We carry out some computational examples to confirm the theoretical results established in Theorem $\ref{t1}$, for solving the FitzHugh-Nagumo model $(\ref{1})$-$(\ref{3})$, and to demonstrate the accuracy of the proposed numerical technique $(\ref{s1})$-$(\ref{s4})$. Suppose $m=4$, $p=6$, and let $\tau_{N}\in\{2^{-l},\text{\,\,}l=4,5,...,8\}$ be the local time steps and $h=2^{-l}$, for $l=2,3,..,6$, be the space size. The following $L^{\infty}(0,T;[L^{2}(\Omega)]^{2})$-norm is used for computing the errors
     \begin{equation*}
       \||e(h,\tau_{N})|\|_{.,\infty}=\underset{0\leq n\leq N}{\max}\|w^{n}_{h}-w^{n}\|_{.},
       \end{equation*}
       where $w_{h}^{n}=(u_{h}^{n},v_{h}^{n})^{t}\in \mathcal{W}_{h}$, is the approximate solution provided by the new algorithm $(\ref{s1})$-$(\ref{s4})$ at time level $n$, whereas $w\in H^{4}(0,T;[H^{4}(\Omega)]^{2}\cap[L_{\sup}^{2}(\Omega)]^{2})$ denotes the analytical solution of the initial-boundary value problem $(\ref{1})$-$(\ref{3})$. The convergence order $CO(h,\tau_{N})$ of the constructed strategy is estimated using the formula
       \begin{equation*}
        CO(h,\tau_{N})=\log_{2}\left(\frac{\||e(2h,\tau_{N})|\|_{.,\infty}}{\||e(h,\tau_{N})|\|_{.,\infty}}\right),\text{\,\,\,\,and\,\,\,\,}CO(h,\tau_{N})=
        \log_{2}\left(\frac{\||e(h,2\tau_{N})|\|_{.,\infty}}{\||e(h,\tau_{N})|\|_{.,\infty}}\right).
       \end{equation*}
        Finally, the numerical computations are performed using MATLAB R$2007b$.\\

      $\bullet$ \textbf{Example $1$}. Suppose that $\Omega=(0,\pi)\times(0,\pi)$, is the fluid region and let $T=1$, be the final time. We consider the following FitzHugh-Nagumo system defined in \cite{24wcqh} by
     \begin{equation*}
      \left\{
        \begin{array}{ll}
          u_{t}-\Delta u=u(1-u)(u-0.5)-v+f_{1}(x,y,t,u), & \hbox{on $\Omega\times[0,\text{\,}1]$} \\
          v_{t}=u-v+f_{2}(x,y,t), & \hbox{on $\Omega\times[0,\text{\,}1]$}
        \end{array}
      \right.
     \end{equation*}
     whose the exact solution $w=(u,v)^{t}$ is defined as
      \begin{equation*}
     u(x,y,t)=(t^{3}+1)\sin(2x)\sin(2y),\text{\,\,\,\,\,}v(x,y,t)=(t^{3}+1)\sin(x)\sin(y).
     \end{equation*}
    The initial and boundary conditions are determined from the analytical solution, while the source terms $f_{k}$, for $k=1,2$, are given by
\begin{eqnarray*}
  f_{1}(x,y,t,u) &=& (8t^{3}+3t^{2}+8)\sin(2x)\sin(2y)-u(1-u)(u-0.5)+(t^{3}+1)\sin(x)\sin(y), \\
  f_{2}(x,y,t) &=& (t^{3}+3t^{2}+1)\sin(x)\sin(y)-(t^{3}+1)\sin(2x)\sin(2y).
\end{eqnarray*}

       \textbf{Table 1} $\label{T1}$. Stability and convergence order $CO(h,\tau_{N})$ of the new predictor-corrector approach with orthogonal spline collocation finite element scheme with varying space step $h$ and time step $\tau_{N}$.
          \begin{equation*}
          \begin{array}{c}
          \text{\,developed scheme:\,\,}\tau_{N}=2^{-6}\\
           \begin{tabular}{|c|c|c|c|c|c|}
            \hline
            $h$  &  $\||e_{u}(h,\tau_{N})|\|_{.,\infty}$ & $CO(h,\tau_{N})$ & $\||e_{v}(h,\tau_{N})|\|_{.,\infty}$ & $CO(h,\tau_{N})$ & CPU (s) \\
             \hline
            $2^{-2}$ & $1.3570\times10^{-3}$ & --- & $1.4291\times10^{-3}$ &  ---    & $3.3011$\\
            \hline
            $2^{-3}$ & $7.7876\times10^{-5}$ & 4.1231 & $8.2076\times10^{-5}$ &  4.1220 & $6.1565$\\
            \hline
            $2^{-4}$ & $4.4621\times10^{-6}$ & 4.1254 & $4.7259\times10^{-6}$ &  4.1183 & $14.4063$\\
            \hline
            $2^{-5}$ & $2.4265\times10^{-7}$ & 4.2008 & $2.5520\times10^{-7}$ &  4.2109 & $31.5426$\\
            \hline
            $2^{-6}$ & $1.2825\times10^{-8}$ & 4.2419 & $1.3500\times10^{-8}$&  4.2406 & $73.3996$\\
           \hline
          \end{tabular}
           \end{array}
            \end{equation*}

          \begin{equation*}
          \begin{array}{c}
          \text{\,developed scheme:\,\,}h=2^{-4}\\
         \begin{tabular}{|c|c|c|c|c|c|}
            \hline
            $\tau_{N}$ &  $\||e_{u}(h,\tau_{N})|\|_{.,\infty}$ & $CO(h,\tau_{N})$ & $\||e_{v}(h,\tau_{N})|\|_{.,\infty}$ & $CO(h,\tau_{N})$ & CPU (s) \\
             \hline
            $2^{-4}$ & $2.8871\times10^{-3}$ & --- & $3.0089\times10^{-3}$ &   --    & $2.7194$\\
            \hline
            $2^{-5}$ & $7.2629\times10^{-4}$ & 1.9910 & $7.5814\times10^{-4}$ & 1.9887  & $6.1072$\\
            \hline
            $2^{-6}$ & $1.7998\times10^{-4}$ & 2.0127 & $1.8908\times10^{-4}$& 2.0035  & $14.4063$\\
            \hline
            $2^{-7}$ & $4.3856\times10^{-5}$ & 2.0370 &  $4.6908\times10^{-5}$ & 2.0111 & $33.9528$\\
            \hline
            $2^{-8}$ & $1.0012\times10^{-5}$ & 2.1310 & $1.0860\times10^{-5}$ & 2.1108 & $82.6954$\\
           \hline
          \end{tabular}
          \end{array}
            \end{equation*}
          \text{\,}\\
          \text{\,}\\
          $\bullet$ \textbf{Example $2$}. Let $\Omega=(-1,\text{\,}1)^{2}$, be the fluid region and $[0,T]=[0,2]$, be the time interval. Consider the initial-boundary value problem given in \cite{8dm} by
       \begin{equation*}
      \left\{
        \begin{array}{ll}
          u_{t}=u(1-u)(u-0.5)-v+f_{1}(x,y,t), & \hbox{on $\Omega\times[0,\text{\,}2]$} \\
          v_{t}-\Delta v=u-v+f_{2}(x,y,t). & \hbox{on $\Omega\times[0,\text{\,}2]$}
        \end{array}
      \right.
     \end{equation*}
     The analytical solution $w=(u,v)^{t}$ is defined as
      \begin{equation*}
     u(x,y,t)=\exp(t)\cos(\pi x)\cos(3\pi y),\text{\,\,\,\,\,}v(x,y,t)=\exp(2t)\cos(2\pi x)\cos(4\pi y).
     \end{equation*}
    The initial and boundary conditions together with the source terms $f_{k}$, for $k=1,2$, are directly obtained from the analytical solution.\\

         \textbf{Table 2} $\label{T2}$. Stability and convergence order $CO(h,\tau_{N})$ of the new predictor-corrector approach with orthogonal spline collocation finite element scheme with varying space step $h$ and time step $\tau_{N}$.
               \begin{equation*}
          \begin{array}{c }
          \text{\,developed scheme:\,\,}\tau_{N}=2^{-6}\\
           \begin{tabular}{|c|c|c|c|c|c|}
            \hline
            $h$ &   $\||e_{u}(h,\tau_{N})|\|_{.,\infty}$ & $CO(h,\tau_{N})$ & $\||e_{v}(h,\tau_{N})|\|_{.,\infty}$ & $CO(h,\tau_{N})$ & CPU (s) \\
             \hline
            $2^{-2}$ & $6.6520\times10^{-4}$ & --- & $7.2415\times10^{-4}$ &   ---      & $2.8854$\\
            \hline
            $2^{-3}$ & $4.1196\times10^{-5}$ & 4.0132 & $4.5344\times10^{-5}$ &  3.9973 & $6.1915$\\
            \hline
            $2^{-4}$ & $2.4008\times10^{-6}$ & 4.1009 & $2.5144\times10^{-6}$ &  4.1726 & $12.3719$\\
            \hline
            $2^{-5}$ & $1.4997\times10^{-7}$ & 4.0008 & $1.4664\times10^{-7}$ &  4.0999 & $29.0195$\\
            \hline
            $2^{-6}$ & $8.5323\times10^{-9}$ & 4.1356 & $8.0793\times10^{-9}$&  4.1819  & $72.2383$\\
           \hline
          \end{tabular}
          \end{array}
            \end{equation*}
          \begin{equation*}
          \begin{array}{c }
          \text{\,developed scheme:\,\,}h=2^{-4}\\
           \begin{tabular}{|c|c|c|c|c|c|}
            \hline
            $\tau_{N}$ &  $\||e_{u}(h,\tau_{N})|\|_{.,\infty}$ & $CO(h,\tau_{N})$ & $\||e_{v}(h,\tau_{N})|\|_{.,\infty}$ & $CO(h,\tau_{N})$ & CPU (s) \\
             \hline
            $2^{-4}$ & $4.2187\times10^{-3}$ & --- & $2.5403\times10^{-3}$ &   --    & $2.9221$\\
            \hline
            $2^{-5}$ & $1.0566\times10^{-3}$ & 1.9974 & $6.3464\times10^{-4}$ & 2.0010  & $6.2500$\\
            \hline
            $2^{-6}$ & $2.6406\times10^{-4}$ & 2.0005&  $1.5465\times10^{-4}$& 2.0369 & $13.2716$\\
            \hline
            $2^{-7}$ & $5.8844\times10^{-5}$ & 2.1659&  $3.3960\times10^{-5}$ & 2.1871 & $32.5578$\\
            \hline
            $2^{-8}$ & $1.2380\times10^{-5}$ & 2.2489&  $7.0439\times10^{-6}$ & 2.2694 & $76.2471$\\
           \hline
          \end{tabular}
          \end{array}
            \end{equation*}
        \text{\,}\\
          \text{\,}\\
        \textbf{Tables} $1$-$2$ indicate that the proposed predictor-corrector approach $(\ref{s1})$-$(\ref{s4})$ is temporal second-order accurate and spatial fourth-order convergent. Additionally, Figures $\ref{figure1}$-$\ref{figure2}$ suggest that the new computational technique is unconditionally stable. These computational results confirm the theoretical studies provided in Theorem $\ref{t1}$.\\
        \text{\,}\\
          \text{\,}\\
      $\bullet$ \textbf{Example $3$}. Suppose that $\Omega=(0,2.5)\times(0,2.5)$, is the fluid region and let $T=1$, be the final time. We analyze the developed predictor-corrector scheme $(\ref{s1})$-$(\ref{s4})$, with discontinuous initial conditions and concerned with unconditional stability. We consider the problem described in \cite{24wcqh} as
         \begin{equation*}
      \left\{
        \begin{array}{ll}
          u_{t}-\gamma_{1}\Delta u=u(1-u)(u-\theta_{3})-v, & \hbox{on $\Omega\times(0,\text{\,}1]$} \\
          v_{t}=\epsilon_{0}(\theta_{1}u-\theta_{2}v-\theta_{0}), & \hbox{on $\Omega\times(0,\text{\,}1]$}
        \end{array}
      \right.
     \end{equation*}
     where $\theta_{3}=\epsilon_{0}=10^{-2}$, $\gamma_{1}=10^{-4}$, $\theta_{1}=0.5$, $\theta_{2}=1$, and $\theta_{0}=0$. The initial conditions are given by
      \begin{equation*}
     u(x,y,0)=\left\{
                \begin{array}{ll}
                  1, & \hbox{if $0<x\leq1.25$, $0<y\leq 1.25$} \\
                  0, & \hbox{if $1.25\leq x <2.5$, $0<y\leq1.25$} \\
                  0, & \hbox{if $0<x\leq1.25$, $1.25\leq y<2.5$} \\
                  0, & \hbox{if $1.25\leq x<2.5$, $1.25\leq y<2.5$}
                \end{array},
              \right.\text{\,\,\,}v(x,y,0)=\left\{
                \begin{array}{ll}
                  0, & \hbox{if $0<x\leq1.25$, $0<y\leq 1.25$} \\
                  0, & \hbox{if $1.25\leq x <2.5$, $0<y\leq1.25$} \\
                  0.1, & \hbox{if $0<x\leq1.25$, $1.25\leq y<2.5$} \\
                  0.1, & \hbox{if $1.25\leq x<2.5$, $1.25\leq y<2.5$}
                \end{array}
              \right.,
     \end{equation*}
      and boundary conditions
       \begin{equation*}
        u(x,y,t)=0,\text{\,\,\,}v(x,y,t)=0,\text{\,\,\,on\,\,\,}\Gamma\times(0,\text{\,}1].
       \end{equation*}
       Figure $\ref{figure3}$ shows that the developed strategy is unconditionally stable for problems with discontinuity initial conditions. Finally, the analysis discussed in this section indicates that the new numerical approach $(\ref{s1})$-$(\ref{s4})$ for solving the initial-boundary value problem $(\ref{1})$-$(\ref{3})$, calculates efficiently the predicted and corrected solutions and maintains strong stability and high-accuracy even in the presence of singularities.

     \section{General conclusions and future investigations}\label{sec5}
     In this paper, we have proposed a high-order predictor-corrector approach with orthogonal spline collocation finite element method for solving a FitzHugh-Nagumo model $(\ref{1})$ subject to suitable initial-boundary conditions $(\ref{2})$-$(\ref{3})$. The developed numerical method approximates the analytical solution in time using variable time steps in predictor phase and a constant time step in corrector stage while the orthogonal spline collocation finite element methods are used in the space discretization. As a result, the new algorithm has several advantages: (a) the errors increased at the predictor stage are balanced by the ones decreased at the corrector phase so that the stability of the numerical scheme is maintained, (b) the variable time steps at the predictor stage considerably reduces the numerical oscillations, (c) the use of the collocation nodes increases the number of traingles/tetrahedra which increases the dimension of the approximation space and thus, minimizes the spatial errors, and (d) the linearization of the nonlinear term minimizes the required iterations at the corrector stage. Both theoretical and numerical studies have indicated that the new algorithm $(\ref{s1})$-$(\ref{s4})$ is unconditionally stable, spatial fourth-order accurate and temporal second-order convergent in the $L^{\infty}(0,T;[H^{m}]^{2})$-norm. Additionally, the new computational technique computes efficiently both predicted and corrected solutions and preserves a strong convergence and high-order accuracy in the presence of singularities. Our future works will construct a predictor-corrector scheme combined with an orthogonal spline collocation finite element method for solving the two-dimensional parabolic interface problems.

         \subsection*{Ethical Approval}
          Not applicable.
         \subsection*{Availability of supporting data}
          Not applicable.
         \subsection*{Declaration of Interest Statement}
          The author declares that he has no conflict of interests.
         \subsection*{Authors' contributions}
          The whole work has been carried out by the author.
         \subsection*{Funding}
          Not applicable.

    \newpage

        \begin{figure}
         \begin{center}
        Stability and convergence of the predictor-corrector approach with orthogonal spline collocation FEM.
         \begin{tabular}{c c}
         \psfig{file=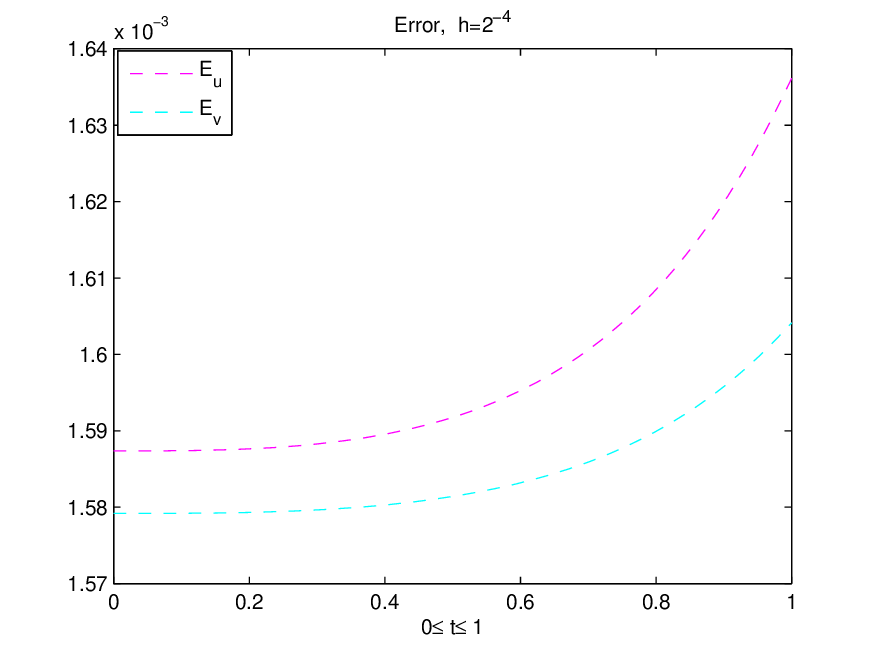,width=6cm} & \psfig{file=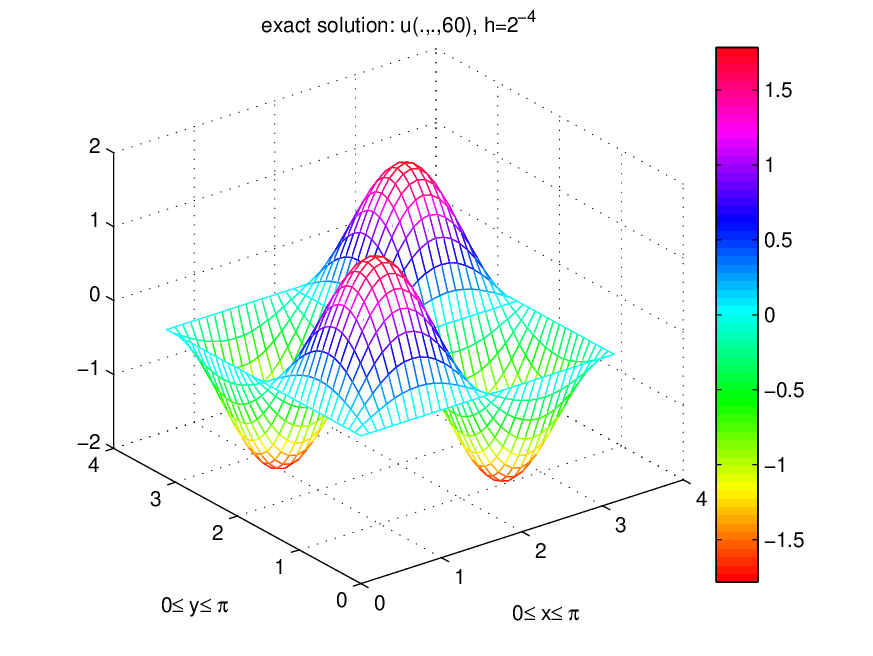,width=6cm}\\
         \psfig{file=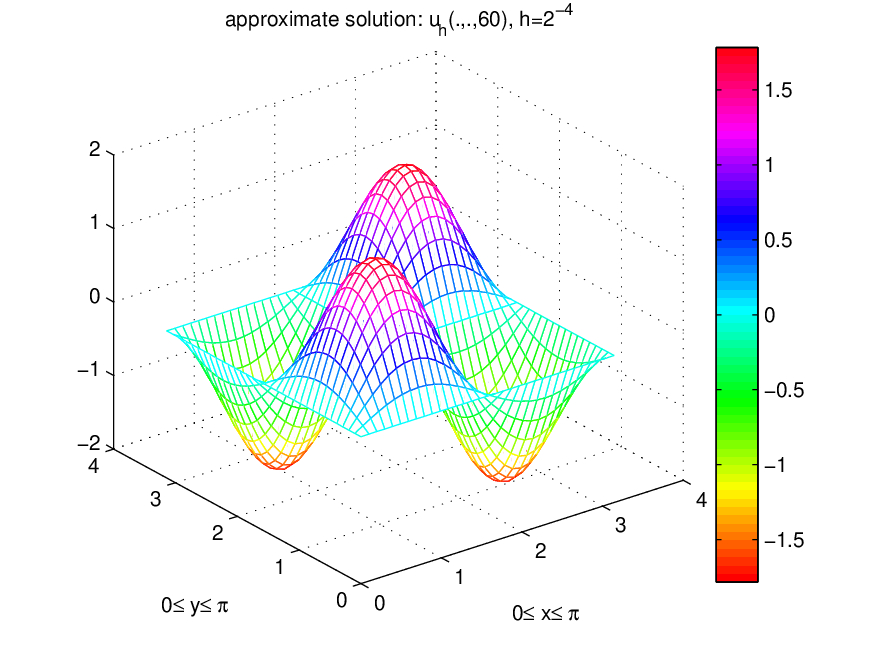,width=6cm} & \psfig{file=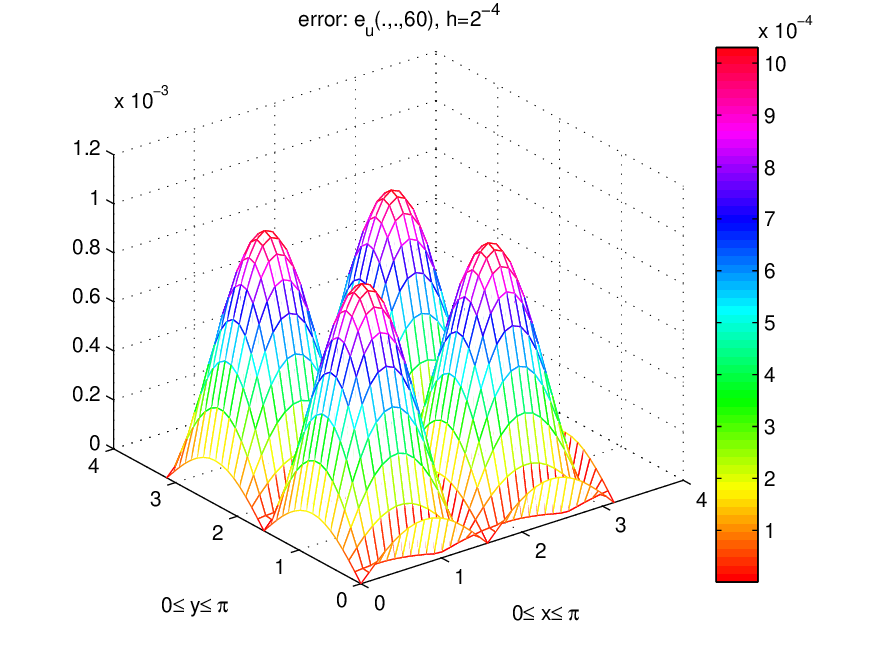,width=6cm}\\
         \psfig{file=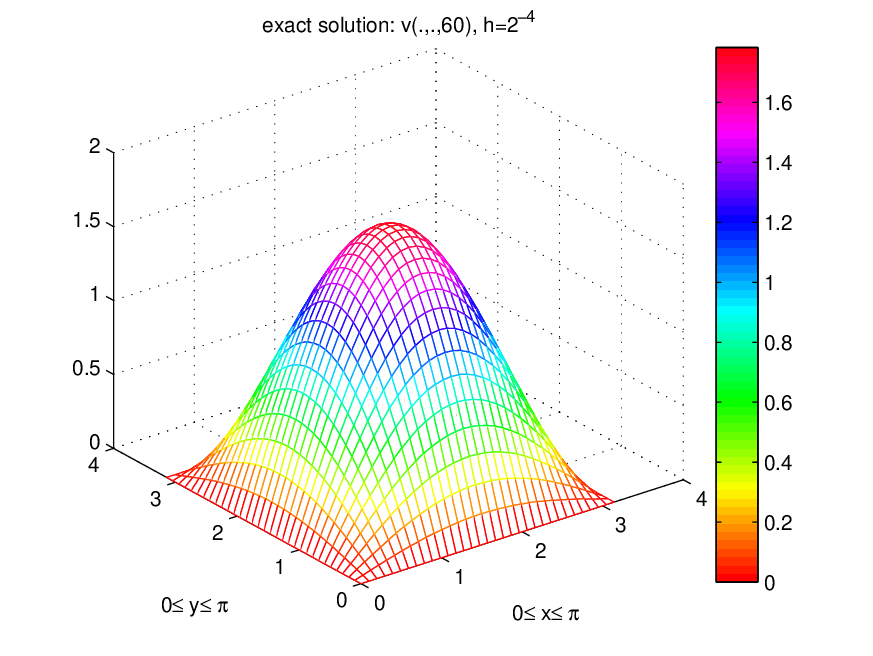,width=6cm} & \psfig{file=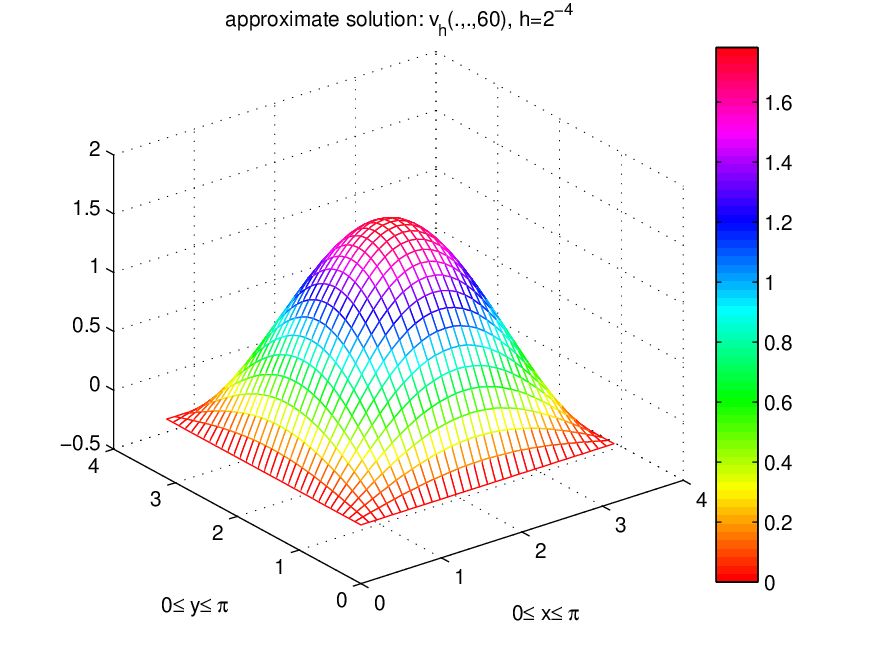,width=6cm}\\
         \psfig{file=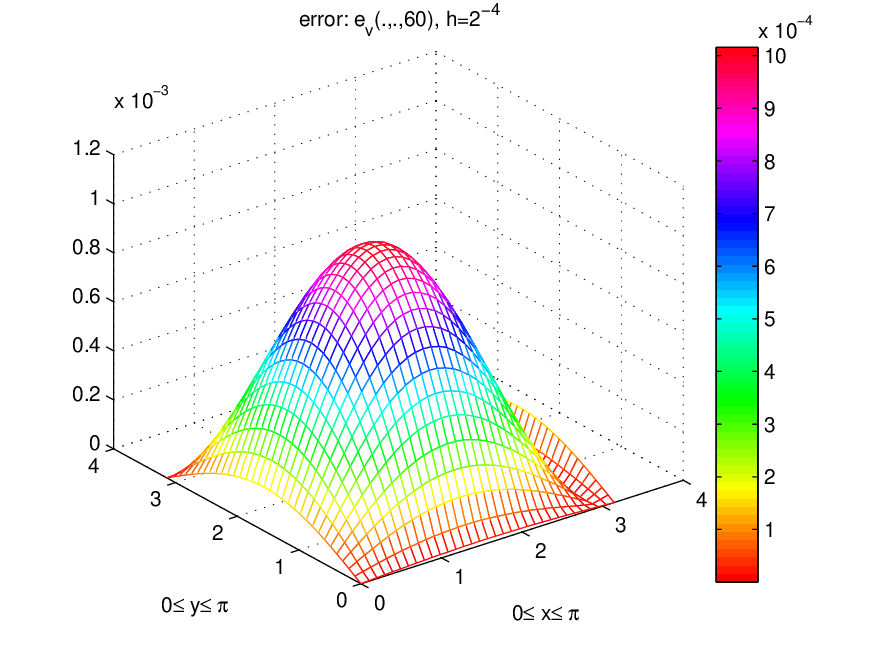,width=6cm} & \\
         \end{tabular}
        \end{center}
        \caption{exact solution, approximate solution and error corresponding to Example 1}
        \label{figure1}
        \end{figure}

           \begin{figure}
         \begin{center}
         Stability and convergence of the predictor-corrector approach with orthogonal spline collocation FEM.
         \begin{tabular}{c c}
         \psfig{file=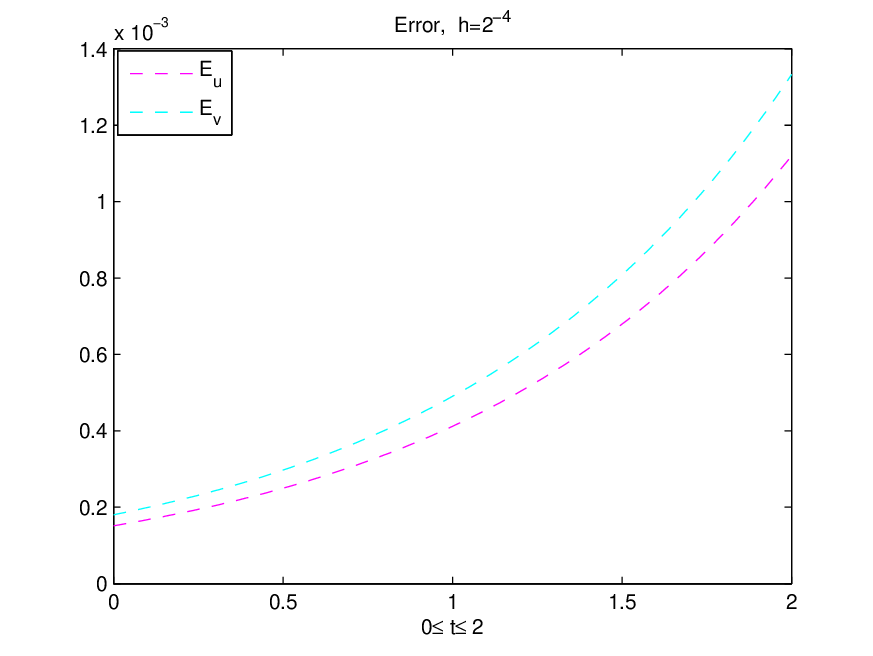,width=6cm} & \psfig{file=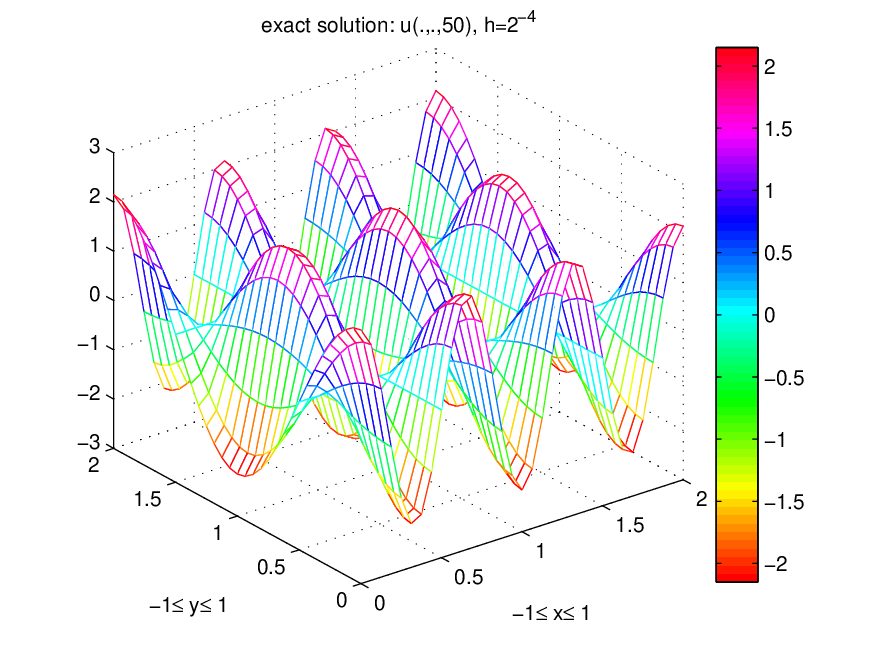,width=6cm}\\
         \psfig{file=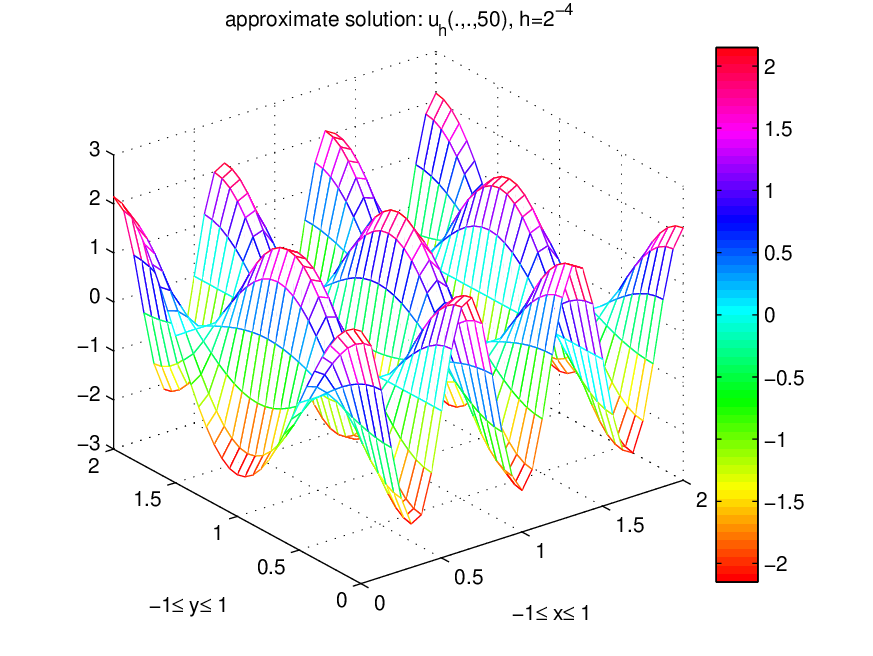,width=6cm} & \psfig{file=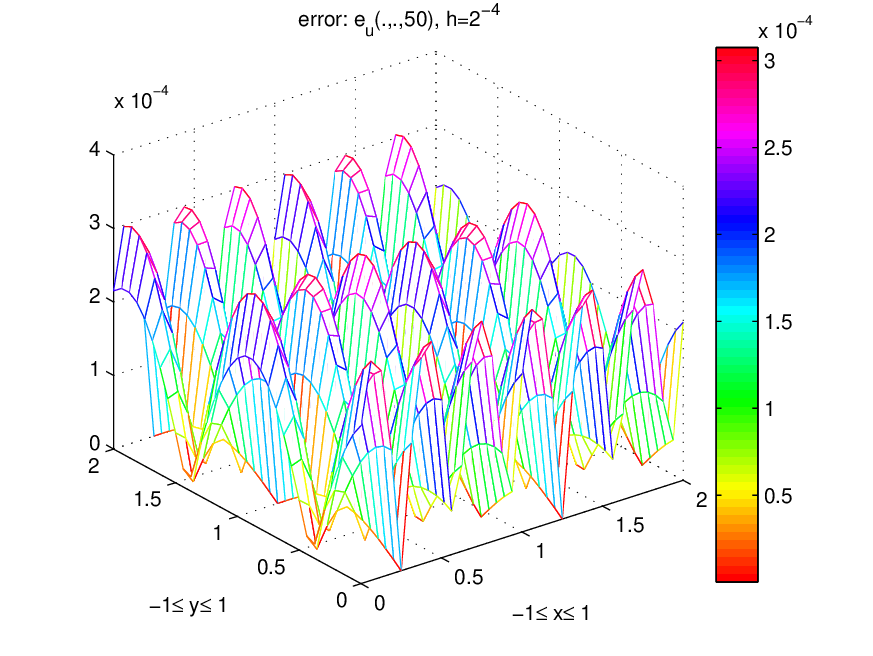,width=6cm}\\
         \psfig{file=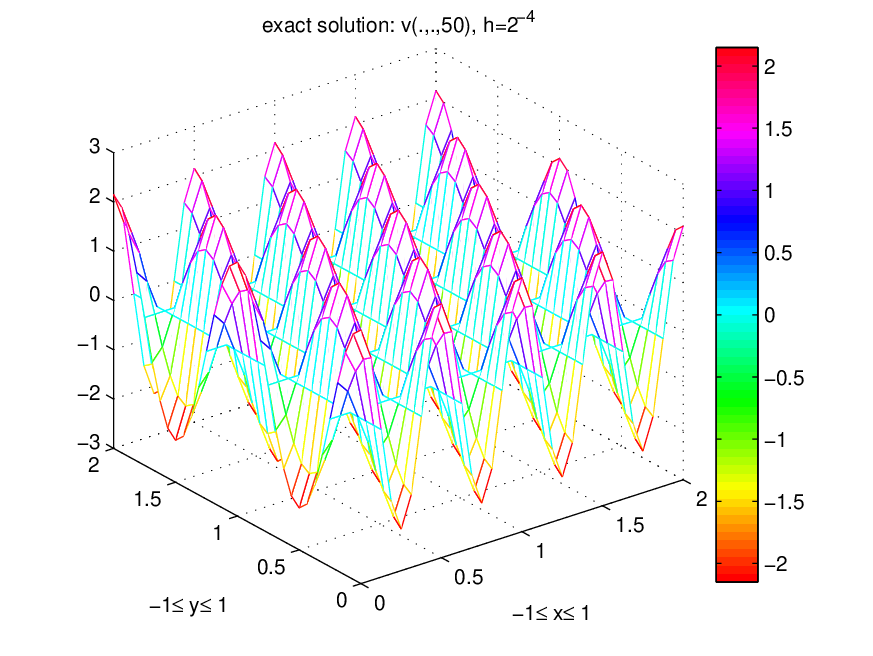,width=6cm} & \psfig{file=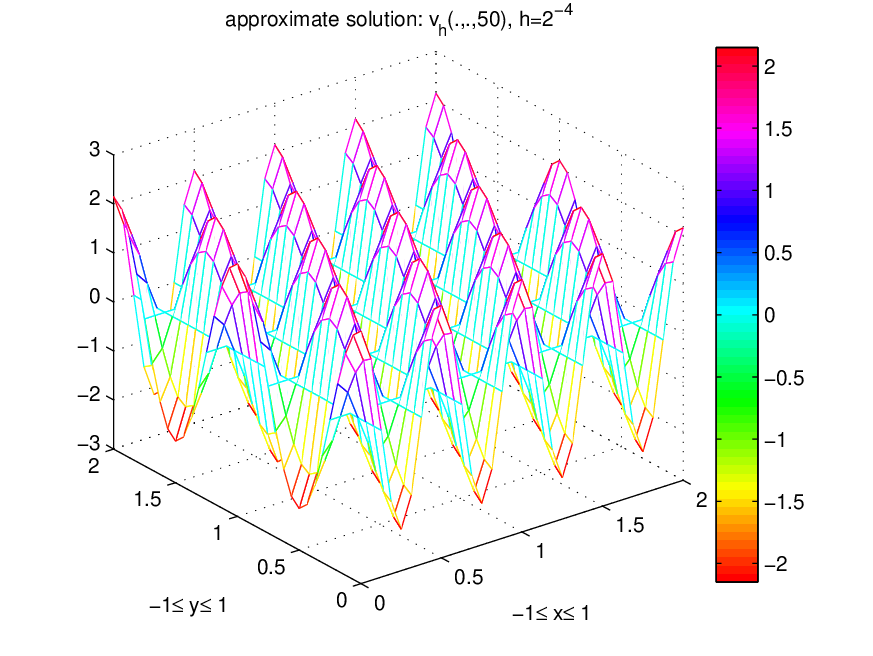,width=6cm}\\
         \psfig{file=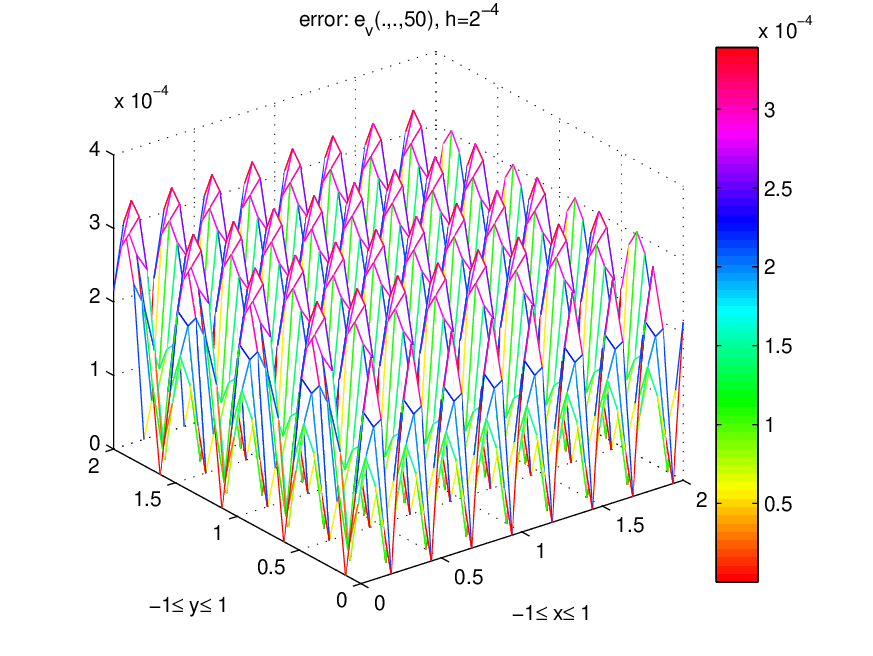,width=6cm} & \\
         \end{tabular}
        \end{center}
        \caption{exact solution, approximate solution and error associated with Example 2}
        \label{figure2}
        \end{figure}

       \begin{figure}
         \begin{center}
        Stability of the predictor-corrector approach with orthogonal spline collocation FEM.
         \begin{tabular}{c c}
         \psfig{file=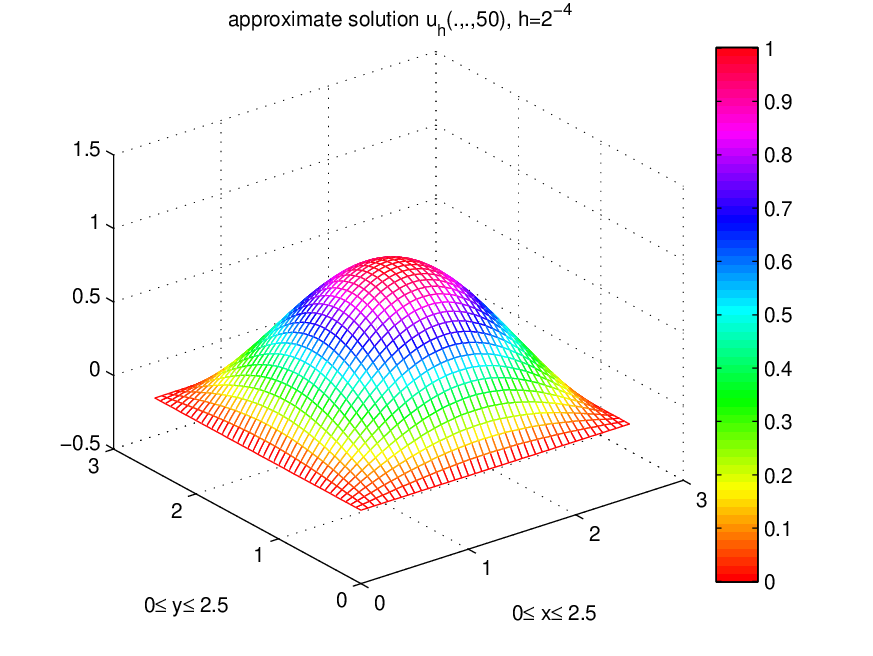,width=6cm} & \psfig{file=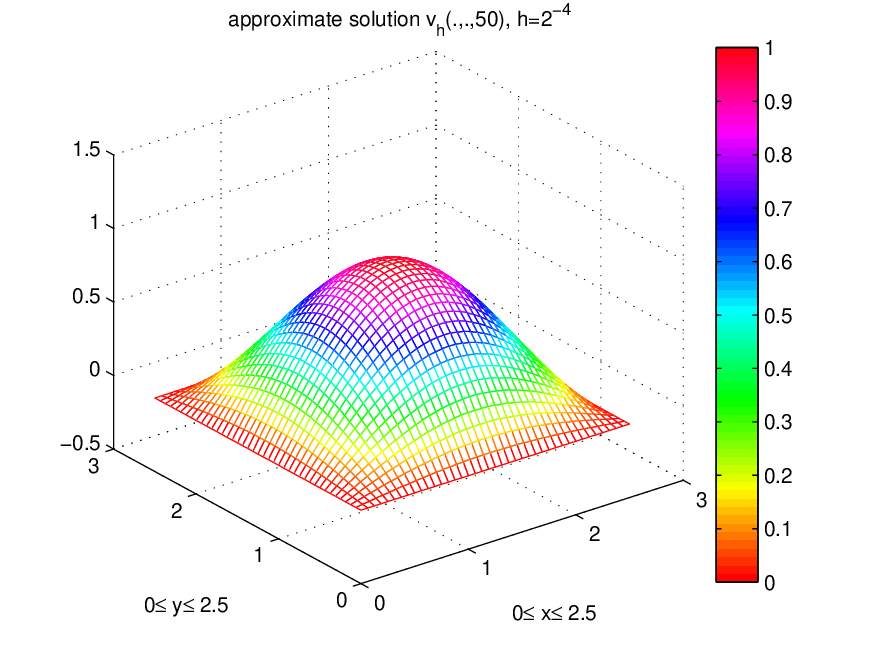,width=6cm}\\
         \psfig{file=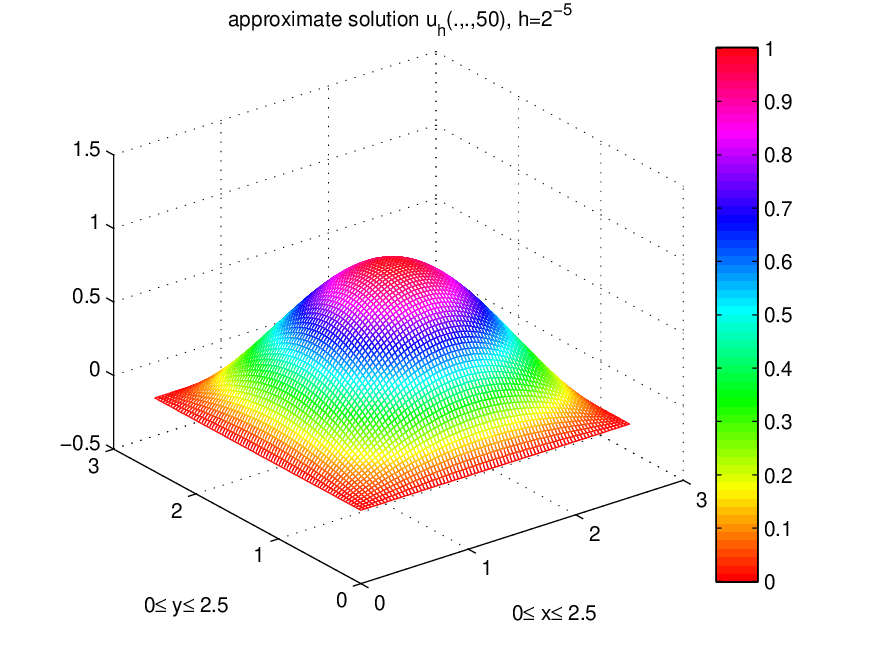,width=6cm} & \psfig{file=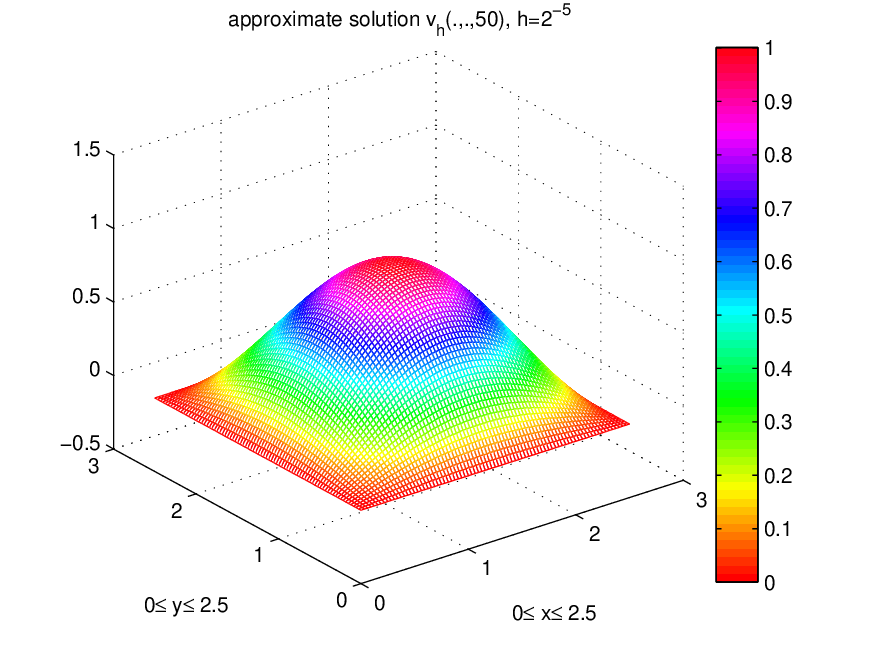,width=6cm}\\
         \psfig{file=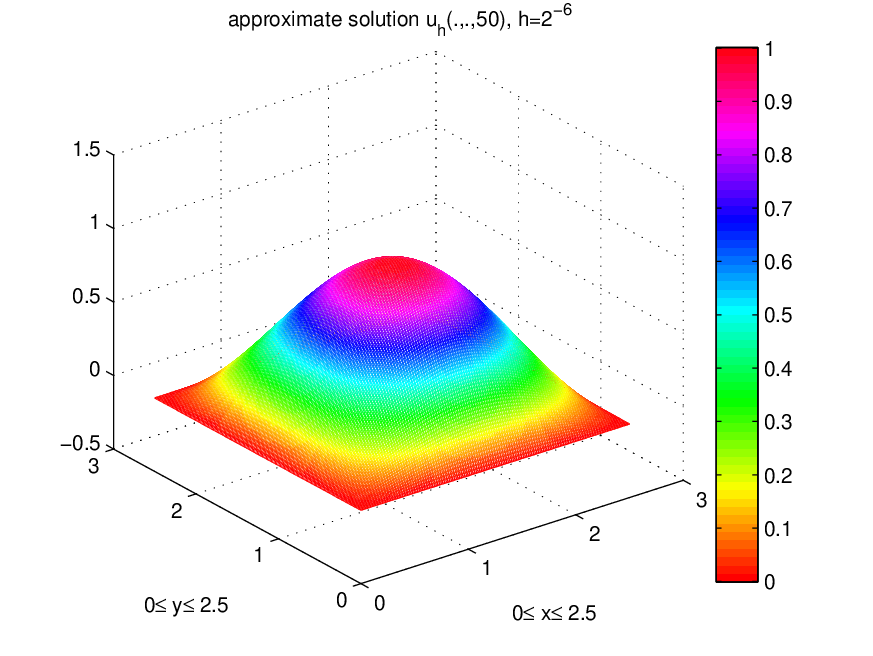,width=6cm} & \psfig{file=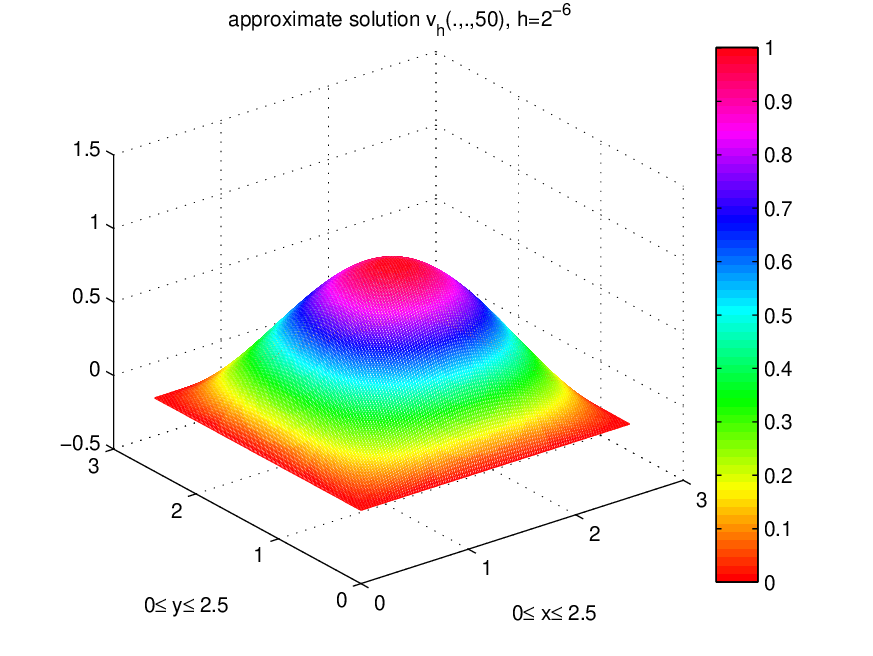,width=6cm}\\
         \end{tabular}
        \end{center}
        \caption{numerical solution associated with Example 3}
        \label{figure3}
        \end{figure}

       \end{document}